\documentclass[a4paper,11pt,reqno]{amsart}
\usepackage{
amsfonts,
amsmath,
amsopn,
amssymb,
amsthm,
bbm,
bbold,
comment,
dsfont,
enumitem,
graphicx,
mathrsfs,
mathtools,
mathabx,
soul,
subfig,
verbatim,
xcolor,
xspace,
latexsym
}

\usepackage{geometry}
\usepackage[hidelinks]{hyperref}
\usepackage[utf8]{inputenc}
\usepackage{float}

\DeclareGraphicsExtensions{.png, .pdf}
\DeclareMathOperator\arctanh{arctanh}

\usepackage{wrapfig}
\usepackage{lipsum}
\usepackage{caption}

\usepackage{multicol} % This is so we can have multiple columns of text side-by-side
\usepackage{rotating}

\usepackage{tikz}
\usetikzlibrary{decorations.pathreplacing}
\usetikzlibrary{shapes.geometric,calc,decorations.markings,math}

\mathtoolsset{showonlyrefs}

\tikzstyle{nodino}=[circle,draw,fill,inner sep=0pt,minimum size=0.5mm]
\tikzstyle{infinito}=[circle,inner sep=0pt,minimum size=0mm]
\tikzstyle{nodo}=[circle,draw,fill,inner sep=0pt, minimum size=0.5*width("k")]
\tikzstyle{nodo_vuoto}=[circle,draw,inner sep=0pt, minimum size=0.5*width("k")]
\usetikzlibrary{graphs}
\tikzset{every loop/.style={min distance=10mm,in=300,out=240,looseness=10}}
\tikzset{place/.style={circle,thick,draw=blue!75,fill=blue!20,minimum
		size=6mm}}
\tikzset{place2/.style={circle,thick,draw=red!75,fill=red!20,minimum
		size=6mm}}

\usepackage{xcolor}

\definecolor{DEblue}{rgb}{0.08,0.24,0.54}
\definecolor{DEgreen}{rgb}{0.07,0.76,0.75}
\definecolor{DEorange}{rgb}{0.92,0.50,0.04}		
\DeclareMathOperator{\sech}{sech}
\newcommand{\NLS}{\rm{NLS}}
\newcommand{\x}{{\bf x}}
\newcommand{\R}{{\mathbb R}}
\newcommand{\Rp}{\R^{+}}
\newcommand{\Rm}{\R^{-}}

\newcommand{\E}{{\mathcal{E}}}
\newcommand{\ep}{\varepsilon}
\newcommand{\dip}{\eta^{dip}}

\newcommand{\Dtau}{\mathcal{D}^\tau}
\newcommand{\Dtaumu}{\mathcal{D}^{\tau}_{\mu}}

\newcommand{\om}{\omega}
\newcommand{\la}{\lambda}

\newcommand{\deb}{\rightharpoonup}

\newcommand{\dx}{\,dx}
\newcommand{\dt}{\,dt}
\newcommand{\ds}{\,ds}

\newcommand{\intRneg}{\int_{-\infty}^{0}}
\newcommand{\intRpos}{\int_{0}^{+\infty}}

\newcommand{\intone}{\int_{-1}^{1}}

\newcommand{\f}[2]{\frac{#1}{#2}}

\theoremstyle{plain} 
\newtheorem{theorem}{Theorem}[section] 
\newtheorem{corollary}[theorem]{Corollary} 
\newtheorem{lemma}[theorem]{Lemma} 
\newtheorem{proposition}[theorem]{Proposition} 

\theoremstyle{definition}

\theoremstyle{definition} 
\newtheorem{definition}{Definition}[section]

\theoremstyle{remark} 
\newtheorem{remark}{Remark}[section]

\title[An explicitly solvable NLS model with discontinuous standing waves]{An explicitly solvable NLS model\\  with discontinuous standing waves}

\author[R. Adami]{Riccardo Adami}
\address{R. Adami: Politecnico di Torino, Dipartimento di Scienze Matematiche ``G.L. Lagrange'', Corso Duca degli Abruzzi, 24, 10129, Torino, Italy.}
\email{riccardo.adami@polito.it}

\author[F. Boni]{Filippo Boni}
\address{F. Boni: Scuola Superiore Meridionale, Largo S. Marcellino, 10, 80138, Napoli, Italy.}
\email{f.boni@ssmeridionale.it}

\author[T. Nakamura]{Takaaki Nakamura}
\address{T. Nakamura: Kochi University of Technology, 185 Tosayamadacho Miyanokuchi, Kami, Kochi 782-8502, Japan}
\email{1160229d@gmail.com}

\author[A. Ruighi]{Alice Ruighi}
\address{A. Ruighi: Istituto di Istruzione Superiore "Cravetta Marconi", Corso Roma, 70, 12038, Savigliano (CN), Italy.}
\email{aliceruighi@gmail.com}

\begin{document}
\begin{abstract}
   {We study the NLS Equation on the line with a point interaction given by the superposition of an attractive delta potential with a dipole interaction, in the cases of $L^2$-subcritical and $L^2$-critical nonlinearity.
   	
   	For a subcritical nonlinearity we prove the existence and the uniqueness of Ground States at any mass. If the mass exceeds an explicit threshold, then there exists a positive excited state too.
   	
   	For the critical nonlinearity we prove that Ground States exist only in a specific interval of masses, while in a different interval excited states exist. We provide the value of the optimal constant in the Gagliardo-Nirenberg estimate and describe in the dipole case the branches of the stationary states as the strength of the interaction varies. 
   	
   	Since all stationary states are explicitly computed, ours is a solvable model involving a non-standard interplay of a nonlinearity with a point interaction.

   }
\end{abstract}

\maketitle

\vspace{-.5cm}

\noindent {\footnotesize \textul{Keywords:} standing waves, nonlinear Schr\"odinger, ground states, delta interaction, critical nonlinearity, excited states, bifurcation, positive solutions, rearrangements.}

       \section{Introduction}
       \label{sec:intro}
       Pointwise perturbations of the Laplacian are nowadays a well established research topic in various branches of PDEs and Mathematical Physics. They are used to model the effect of microscopic impurities 
       on electromagnetic, acoustic or quantum mechanical phenomena.
       %in which case they are referred to as point interactions.
        In particular, the interaction of 
        impurities with 
        Bose-Einstein condensates (\cite{LSSY,PS}) has often been modeled
        by an additional term in the energy functional, leading to the
        following modified
       Gross-Pitaveskii energy (\cite{G61,P63})
        \begin{equation}
      \label{EGP}  
    E_{\rm{GP}}(\Phi) = \frac{1}{2}\|\nabla\Phi\|_{L^2(\R^3)}^2 + 2 \pi \ell \| \Phi \|_{L^4(\R^3)}^4  + \mathcal{I} (\Phi),
    \end{equation} 
       whose minimizers
      under the constraint 
      \begin{equation} \label{constraint1}
\|\Phi\|_{L^2(\R^3)}^2=N
      \end{equation}
      describe the spatial
      distribution %of the $N$ particles 
      of the condensate.
      
      In \eqref{EGP} the function $\Phi$ is the wave function of the condensate,
 $\ell$ is the scattering length of the 
interaction between the particles of the condensate, and $\mathcal{I} (\phi)$ is the 
contribution to the energy 
due to the pointwise perturbation. In \eqref{constraint1}, $N$ is the number of particles in the condensate.

       Here we restrict to a one-dimensional setting, namely we consider the dynamics of a particle on a line, subject to a pointwise perturbation located at the origin. We analyse a
       specific family of point interactions, characterised by the two following features:
\begin{itemize}
    \item Any wave function $u$ with a finite energy exhibits a jump discontinuity at the 
    origin, i.e.
         $u(0^+)=\tau u(0^-)$;
     \item
    The contribution of the point perturbation
    to the energy is given by
     $ \mathcal I (u) = - \alpha | u (0^-) |^2/2$. 
     \end{itemize}
     The  range of the real parameters 
     $\tau$ and $\alpha$ will be specified later. Notice that they are independent of $u$, so the resulting pointwise perturbation is linear. 

Concerning the nonlinear term in the energy
functional, we
limit to a focusing nonlinearity and generalize the power law in \eqref{EGP} to include the  $L^2$-critical case, that, as widely known, gives rise to a radically different
phenomenology (see e.g. \cite{C-03}).

Summing up, we 
completely and explicitly solve two problems. The first consists in  giving conditions for  the existence and the nonexistence of Ground States at mass $\mu$, i.e. nonnegative minimizers of the energy
\begin{equation}
\label{eq:energy}
E_\alpha(u):=\frac{1}{2} \left( \|u'\|^2_{L^2(\R^-)}+\|u'\|^2_{L^2(\R^+)} \right) -\frac{1}{2\sigma+2} \|u\|_{L^{2\sigma+2}(\R)}^{2\sigma+2} - \frac{\alpha}{2}|u(0^-)|^2,
\end{equation}
among functions belonging to the space
\begin{equation}
\label{eq:Dtau}
\Dtau : =\{u \in H^1(\R^-) \oplus H^1(\R^+) \, : \, u(0^+)=\tau u(0^-)\,\},
\end{equation}
 satisfying the constraint
\begin{equation}
\label{eq:mass-const}
\|u\|_{L^2(\R)}^2=\mu>0,
\end{equation}
where $0<\sigma\leq 2$, $\alpha\geq 0$ and $\tau>1$. Notice that the presence of the discontinuity forces to split the quantum mechanical kinetic energy, typically identified with the squared
$L^2$-norm of the derivative of the wave function, in the contributions of the positive and of the negative real semiaxis.
Notice also that the choice of looking for nonnegative minimizers does not compromise the generality of the problem: indeed, given a complex valued function $u\in \Dtau$ with $\|u\|_{L^2(\R)}^2=\mu$, it is straightforward to verify that $|u|\in \Dtau$ and $E_\alpha(|u|)\leq E_\alpha(u)$. Moreover, the energy functional and the mass constraint are invariant under 
multiplication by a phase, so that if $u$ is a positive Ground State at some value of the mass,
 then for every real $\theta$ the function $e^{i \theta} u$ is a 
 minimizer of \eqref{eq:energy} under the constraint \eqref{eq:mass-const} too.

 The second problem we solve is the
 existence of other positive stationary states of $E_\alpha$, namely critical points of the energy \eqref{eq:energy} under the constraint \eqref{eq:mass-const} (see Lemma \ref{lem:stat-el}). For this problem
 we explicitly give all standing waves for the system associated with the energy functional
 \eqref{eq:energy}.
 
\smallskip

The choice of the range of the parameters deserves some comments.

\begin{enumerate}
    \item 
The restriction  $\tau >1$  does not yield a loss of generality. Indeed,
 $\tau= 1$ corresponds to the standard NLS energy on the real line with a Dirac's delta interaction,  $\tau=0$ gives  the standard NLS on the halfline  (\cite{bc}), while  $0<\tau<1$  is mapped to   $\tau>1$ by the change  $u\mapsto u(-\cdot)$. Moreover, if $\tau<0$, then the transformation 
\begin{equation*}
u\mapsto \begin{cases}
u,\quad &x< 0\\
-u, \quad &x>0    
\end{cases}
\end{equation*}
 reduces to the case $\tau>0$. 

%Furthermore, once one considers the case $\tau>1$, it is possible to restrict the minimization problem to nonnegative functions $u\in \Dtaumu$ since $E_\alpha(|u|)\leq E_\alpha(u)$ and any ground state is real valued and positive up to a multiplication by a constant phase factor.

\item 
As anticipated, owing to the choice $\sigma\in (0,2]$ we include  the  subcritical ($0<\sigma<2$) as well as the  critical case $\sigma=2$ (Here and in the following by 
critical and subcritical cases we shall always understand the $L^2$-critical and the
$L^2$-subcritical, respectively). %Moreover, the negative sign makes the nonlinearity focusing.

\item The negative sign of the nonlinearity means that we are considering a focusing nonlinearity.

\item 
Due to the hypothesis $\alpha\geq 0$, the point interaction  is attractive.
\end{enumerate}

%Let us now explain why the problem under investigation is that of a pointwise perturbation of the Laplacian in dimension one. 

In agreement with the nomenclature introduced in \cite{anv}, we refer to all cases
with $\alpha = 0$
with the name of dipole interaction, while we keep the name of F\"ul\"op-Tsutsui
interaction for the cases with $\alpha > 0$ (\cite{ft}). In fact, to our knowledge the term 
F\"ul\"op-Tsutsui vertex was coined in \cite{cheon} to designate a vertex in quantum graphs where a generalized point interaction is present. 

One-dimensional models with both a power nonlinearity and a point perturbation of the Laplacian are quite 
popular, yet  still spreading: historically, in the
first models of such kind the perturbation is a standard distributional Dirac's delta \cite{AN-09,bc,fj,foo,HMZ}; later, the so-called delta-prime and dipole interaction were investigated 
\cite{an,anv}, and recently more exotic point
perturbation were scrutinized, including nonlinear $\delta$ conditions \cite{bd}. On the other hand, the coupling between a nonlinearity and a point interaction in more than one dimension has been investigated in a restricted number of studies and only in recent years, e.g. in \cite{ABCT-22,ABCT-22bis,CFN, FN-22,FGI-21}.

Furthermore, in the last years the coupling between a nonlinear Schr\"odinger equation and a pointwise interaction found a natural environment in the context of metric graphs \cite{meh}, where the point interaction is located at the vertices and can be expressed by the free or Kirchhoff's conditions (see e.g. \cite{ast1,ast3,BMP,bcjs,cm,cds,cjs,dst,gsd,good,marpel,serraten}) but 
also by interacting linear conditions
 of the $\delta$-type  \cite{ACFN-14}, or even more singular
\cite{abr}.

\begin{comment}
As already mentioned, we do not treat the Dirac's delta point interaction since its 
interplay with a power nonlinearity was extensively studied by several authors (e.g.
\cite{AN-09,anv,bc,foo,fj,HMZ}). Another point interaction involving a discontinuity at the
origin was investigated in \cite{an,anv}, namely the delta prime interaction, which is not covered by the present analysis. Moreover, the $L^2$-subcritical case in the presence of a dipole interaction ($\alpha=0$) has been already discussed in \cite{anv}, so we shall not
treat it here.
\end{comment}

%As already detailed, in the present paper we give an exhaustive analysis of all stationary states for the model \eqref{eq:energy} in the subcritical and in the critical case. 

\smallskip

Let us informally summarize here  the  results in the present paper. 

First, the subcritical picture turns out to be fairly standard for what concerns the Ground States, in that a positive Ground State exists and is unique for every value of the mass. The same happens in the two cases of the dipole and of the attractive delta separately. In fact, a non-standard feature  for the dipole interaction is the occurrence of a second branch of positive stationary states, i.e. excited states, which are  present for every mass, while for the delta interaction they are never present. Furthermore, superposing the two interactions in a F\"ulop-Tsutsui interaction produces an intermediate behaviour in which an excited state exists above a mass threshold $\mu_\alpha$ only (see Thm. \ref{thm:gs-sub}), and such threshold grows with the strength $\alpha$ of the delta interaction. 

In this respect, it is
worth pointing out that, if one restricts to attractive point interactions, then the discontinuity in the energy space is necessary in order to
produce a second branch of excited states. Indeed, as shown in Proposition 2.1 of \cite{AN-09}, the only boundary conditions
that allow for an energy space made of continuous functions are  the delta ones, that
exhibit the  branch of the Ground States  only (see \cite{foo,fj,anv}). Let us recall that
it is however possible to have continuous excited states for attractive point interactions
whose energy space contains discontinuous functions: this is the case for the $\delta'$ 
interactions \cite{an,anv}. Moreover, if one extends the analysis to non-attractive 
point interactions, then only the free, the Dirichlet and the delta boundary conditions provide an energy space of continuous functions. Then, in the repulsive delta case, the only branch of stationary states is
made of excited states \cite{fj}.

The picture drastically changes for the critical case $\sigma = 2.$ More specifically, for 
a dipole interaction all Ground States share the same mass $\mu^\star$, that is the analogue of the critical mass for the standard NLS. Yet, consistently with the dipole nature of the interaction, there is a second family of stationary states, that we call excited states again, all at the same mass $\widetilde \mu > \mu^\star$. In other words, the concentration on a single mass sphere holds
both for the Ground States and for the excited ones (of course, the two resulting spheres are different). Details are given in Thm. \ref{thm:gs-crit-dip}, together with a pictorial
interpretation of the branches of the Ground States and of the excited states and a full
explanation of the mass thresholds at $\tau = \infty$, namely $\frac{\sqrt 3}{2} \pi$ and 
$\frac{3 \sqrt 3}{2} \pi$ (Sec. \ref{sec:mainres}).

Now, considering a F\"ulop-Tsutsui, i.e. turning an attractive delta interaction on, it happens that  both manifolds of the Ground States and of the excited states spread on
some interval of masses: Ground States exist for every mass below $\mu^\star$, while excited
states are present for every mass between $\frac{\sqrt 3}{2} \pi$ and $\widetilde \mu.$ This
is the content of Thm. \ref{thm:gs-crit}.

As well-known, a dual analysis of the stationary states can be made by means of the study of Euler-Lagrange equation associated to the functional $E_\alpha$. This approach provides the notion of the frequency of the associated standing
waves, identified with the nonlinear eigenvalue or Lagrange multiplier, and shows
that in the positive and the negative half-lines every stationary state must be a piece of soliton. In particular, it provides the boundary conditions that every
stationary state must fulfil (see Lemma \ref{lem:stat-el}),
showing that
the quadratic part of the functional \eqref{eq:energy} is associated to a
self-adjoint operator. This correspondence is pointed out in Remark \ref{rem:operator}.
Then, one can explicitly compute the multiplicity of the set of solutions  in terms of the frequency. This is the content of Sec. \ref{sec:stat}.

We highlight again that all results are explicitly given, that makes  the
present model a particularly transparent example of a highly non-trivial interplay between a
point interaction and a nonlinearity. A delta-prime interaction instead of a F\"ulop-Tsutsui's would probably show a rich
phenomenology too, and we plan to complete the analysis in \cite{an,anv}, limited to the subcritical case, with the study of the
critical case in a short delay.

\medskip

The paper is organized as follows:
\begin{itemize}
\item In Section \ref{sec:mainres} we provide the statements of the main theorems and comment
on their meaning and relevance;
    \item Section \ref{sec:aux-res} presents some auxiliary results, in particular Gagliardo-Nireneberg inequalities in $\Dtau$ and a rearrangement procedure for functions in $\Dtau$;
    \item in Section \ref{sec:subcritical} we give the proof of Theorem \ref{thm:gs-sub} on the existence of Ground States in the $L^2$ subcritical case;
    \item Section \ref{sec:critical} is devoted to the proofs of Theorem \ref{thm:gs-crit-dip} and Theorem \ref{thm:gs-crit}, on the $L^2$ critical case for a dipole and a F\"{u}l\"{o}p-Tsutsui interaction respectively;
    in Section \ref{sec:stat} we deduce the Euler-Lagrange equations (Lemma \ref{lem:stat-el}), study the multiplicity of the set of solutions of the NLS arising from the 
    Euler-Lagrange equation (Proposition \ref{prop:stationary}) and finally explicity
    identify and classify all stationary states (Prop. \ref{prop:stat-sub} and \ref{prop:stat-crit}).
\end{itemize}

\section{Main results}
\label{sec:mainres}

Here we state and describe the results proven in Sec. \ref{sec:subcritical} and \ref{sec:critical}, that constitute the core of the paper. As pointed out in Sec. \ref{sec:intro}, Sec. \ref{sec:stat} complements such results by the explicit characterization of all standing waves.

As a first step we study the subcritical case $0<\sigma<2$ in the presence of a F\"{u}l\"{o}p-Tsutsui interaction $\alpha>0$. 

We have the following

\begin{theorem}[Stationary states for $0<\sigma<2$ and $\alpha>0$]
\label{thm:gs-sub}
If $0<\sigma<2$, $\alpha>0$ and $\tau>1$, then
\begin{enumerate}
\item the infimum of $E_\alpha$ is finite and negative for every $\mu>0$ and there exists a unique positive Ground State for $E_\alpha$ at mass $\mu$.
    \item A further positive stationary state for $E_\alpha$ exists if and only if $\mu>\mu_\alpha$, with
    \begin{equation}
        \mu_\alpha = \frac{(\sigma+1)^{\frac{1}{\sigma}}}{\sigma} \left(\frac{\alpha}{\tau^2-1}\right)^{\frac{2-\sigma}{\sigma}}  \int_{-1}^{1} (1-t^2)^{\frac{1}{\sigma}-1}\dt.
    \end{equation}
\end{enumerate}

\end{theorem}

The proof is carried out through an argument of concentration-compactness,
establishing that
 the existence of Ground States at mass $\mu$ for $E_\alpha$ can be reduced to the problem of finding functions in $\Dtau$ at mass $\mu$ with energy below that
 of the one-dimensional soliton $\varphi_\mu$ 
 with the same mass (details in Sec. 3). This is not  surprising, since
 in several problems of existence of Ground States for the NLS in spatial structures containing
 half-lines (see e.g. \cite{ast1}), the only possible lack of compactness is realized by sequences that escape through
 a half-line mimicking the shape of a soliton. Once established such a point,
 the result in Thm. \ref{thm:gs-sub}
  follows from the presence of the  additional term $-\alpha|u(0^-)|^2$ that lowers the Ground State energy  of a simple dipole interaction, which is already below the energy of $\varphi_\mu$ and has been discussed in \cite{anv}.
%\begin{equation*}
%E_{NLS}(\varphi_\mu)=\frac{1}{2}\|\varphi_\mu'\|_{L^2(\R)}-\frac{1}{2\sigma+2}\|\varphi_\mu\|_{L^{2\sigma+2}(\R)}^{2\sigma+2},
%\end{equation*}
%where $E_{NLS}$ is the standard NLS energy on the real line defined on $H^1$ functions and $\varphi_\mu$ is the one-dimensional soliton, namely the minimizer of $E_{NLS}$ among functions with fixed mass equal to $\mu$. 
Thus, Ground States in the
presence of an attractive F\"{u}l\"{o}p-Tsutsui $\delta$-interaction exist for every value of the mass. 

Point $(2)$ in  Thm. \ref{thm:gs-sub} follows from the analysis of stationary solutions of \eqref{eq:bound-state} carried out in Section \ref{sec:stat} (see Proposition \ref{prop:stat-sub}). Two remarks are in order. First, the presence of two different positive solutions marks an important difference with the standard NLS model in $\R^N$, where the uniqueness of Ground States is proved in \cite{K-89} reducing the problem to a radial setting via the moving planes method \cite{GNN-81}. Second, the appearance of a further positive stationary state is due to the presence of a resonance in the corresponding problem without the nonlinear term. For
more details see Remark \ref{rem:resonance}.

Let us now discuss the  critical case $\sigma=2$. Before stating the main theorems, let us recall that in the  critical case a key role is usually played by Gagliardo-Nirenberg inequality, that in our context (see Lemma \ref{lem:gn-dtau}) reads 
\begin{equation}
\label{eq:gn-crit}
\|u\|_{L^6(\R)}^6\leq K_\tau \left(\|u'\|_{L^2(\Rm)}^2+\|u'\|_{L^2(\Rp)}^2\right)\|u\|_{L^2(\R)}^4\quad\forall\, u\in \Dtau,
\end{equation}
where 
\begin{equation}
K_\tau:=\sup_{u\in\Dtau\setminus\{0\}}\frac{\|u\|_{L^6(\R)}^6}{\left(\|u'\|_{L^2(\Rm)}^2+\|u'\|_{L^2(\Rp)}^2\right)}\|u\|_{L^2(\R)}^4
\end{equation}
is the optimal constant in \eqref{eq:gn-crit}. Plugging \eqref{eq:gn-crit} into the energy $E_0$ and imposing the mass constraint, it follows that
\begin{equation*}
E_0(u)\geq \frac{1}{2}\left(1-\frac{K_\tau}{3}\mu^2\right)\left(\|u'\|_{L^2(\Rm)}^2+\|u'\|_{L^2(\Rp)}^2\right),
\end{equation*}
leading to the definition of the {critical mass}
  \begin{equation}
\label{eq:mu-star}
    \mu^\star:=\sqrt{\frac{3}{K_\tau}}.
\end{equation}  

\noindent
In the critical case too, we start by treating the presence of a pure dipole interaction. 
After
prelimarily introducing the notation
\begin{equation}
\E_\alpha(\mu):=\inf\left\{ E_\alpha(v)\,:\, v\in \Dtau\,,\,\|v\|_{L^2(\R)}^2=\mu\right\},   \label{ealfa}
\qquad \alpha \geq 0,
\end{equation} 
we have the following

\begin{theorem}[Stationary states for $\sigma=2$ and $\alpha=0$]
\label{thm:gs-crit-dip}
If $\sigma=2$, $\tau>1$ and $\mu>0$, then 
\begin{enumerate}
    \item the infimum of the constrained energy $E_0$ is given by \begin{equation}
\label{eq:E0-thm-crit}
    \E_0(\mu)=
    \begin{cases}
        0\quad &\text{if}\quad 0<\mu\leq\mu^\star,\\
        -\infty &\text{if}\quad \mu>\mu^\star.\\
    \end{cases}
\end{equation}
\item Ground states of $E_0$ at mass $\mu$ exist if and only if $\mu=\mu^\star$ and optimize Gagliardo-Nirenberg inequality \eqref{eq:gn-crit}. Moreover, the value of $\mu^\star$ is explicitely known as 
\begin{equation}
\label{eq:mu-star-exact}
\mu^\star=\frac{\sqrt{3}}{2}\left(\frac{\pi}{2}+2 \arcsin\left(\frac{1}{\sqrt{1+\tau^4}}\right)\right).
\end{equation}

\item Further stationary states of $E_0$ exist if and only if $\mu=\widetilde{\mu}$, with 
\begin{equation}
\label{eq:mu-tilde}
\widetilde{\mu}=\frac{\sqrt{3}}{2}\left(\frac{3\pi}{2}-2 \arcsin\left(\frac{1}{\sqrt{1+\tau^4}}\right)\right).
\end{equation}
\end{enumerate}

\end{theorem}

Let us highlight that, as for the  critical case for the standard NLS, Ground States exist at the critical mass $\mu^*$ only and the quantity $\E_0(\mu)$ is bounded and equal to zero if and only if the mass is smaller than the critical mass. 

The first important novelty in Theorem \ref{thm:gs-crit-dip} is that the value of the critical mass is between the values $\frac{\sqrt{3}\pi}{4}$ and $\frac{\sqrt{3}\pi}{2}$, i.e. the critical masses for the standard NLS energy on $\Rp$ and $\R$ respectively and, as a consequence, the optimal Gagliardo-Nirenberg constant $K_\tau$ is between $\frac{4}{\pi^2}$ and $\frac{16}{\pi^2}$, being the optimal constants (see \cite{Nagy-41}) for the standard Gagliardo-Nirenberg inequality on $\R$ and $\Rp$ respectively. 

Moreover, since solutions of \eqref{eq:bound-state} can be computed explicitly, we are also able to compute the exact value of the quantities $\mu^\star$ and $K_\tau$: notice that $\mu^\star= \frac{\sqrt{3}}{2}\pi$ for $\tau=1$ and $\mu^\star\to \frac{\sqrt{3}}{4}\pi$ as $\tau\to+\infty$, thus the standard problems on the line and on the half-line are limiting cases in our analysis.

Furthermore, a remarkable difference with the standard NLS consists in the presence of another family of positive stationary solutions. Such solutions have all the same mass $\widetilde{\mu}$, that is between $\frac{\sqrt{3}}{2}\pi$, the critical mass on the line, and $\frac{3\sqrt{3}}{4}\pi$, the sum of the critical mass on the line and the critical mass on the half-line.

Let us briefly explain why another branch of solutions exists, starting from the standard NLS on the line, that corresponds to the case $\tau=1$. In such situation, there exists only the family of the Ground States, given by the soliton $\varphi_\mu$ and its translations. When the dipole interaction is switched on, i.e. when $\tau$ is close to $1$, a jump discontinuity of the kind
\begin{equation}
 \begin{cases}
u(0^+)=\tau u(0^-),\\
u'(0^+)=\frac{1}{\tau}u'(0^-),
 \end{cases}
\end{equation}
(\cite{anv}) starts to appear, and such discontinuity can be realized by manipulating two different solitons. In particular, for a proper $\xi>0$ the branch of the Ground States arises from the soliton $\varphi_\mu(\cdot-\xi)$, by suitably translating on the left its restriction to $\R^+$ and on the right its restriction to $\R^-$: as a result, the Ground State has mass smaller than the mass of $\varphi_\mu$ and in the limit as $\tau\to+\infty$ its mass is concentrated only on a half-soliton on $\R^+$ (see Figure \ref{fig:GS}) and equals $\frac{\sqrt 3} 2 \pi$. On the other hand, the other branch of stationary states originates from the soliton $\varphi_\mu(\cdot+\xi)$, by translating on the right its restriction to $\R^+$ and on the left its restriction to $\R^-$: by doing this, such stationary state has mass greater than the mass of $\varphi_\mu$ and in the limit as $\tau\to+\infty$ it recovers a half-soliton on $\R^+$ and a whole soliton escaping at infinity on $\R^-$ (see Figure \ref{fig:stat-states}), approaching for the mass the value  $\frac{3 \sqrt 3} 2 \pi$. 

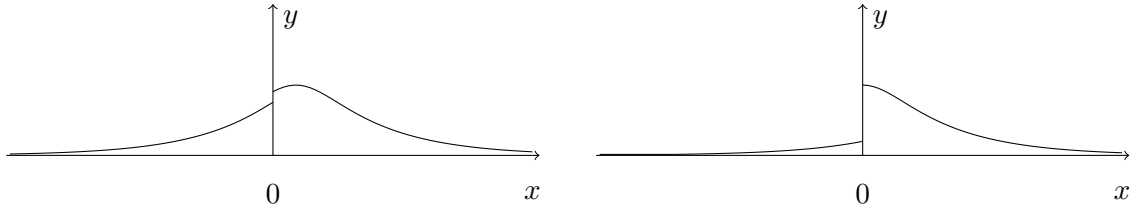
\begin{figure}[h]
\begin{tikzpicture}[xscale= 0.47,yscale=1]
    \draw[step=2,thin,->] (-7.5,0) -- (7.5,0);
    \draw[step=2,thin,->] (0,0) -- (0,2);
\node at (0,-0.5) [] {$0$};
\node at (7.3,-0.5) [] {$x$};
\node at (0.5,1.8) [] {$y$};
\draw[color=black] plot[samples=50,domain=-7.4:0](\x, {0.93 / sqrt(cosh( \x - 1.161 )))});
\draw[color=black] plot[samples=50,domain=0:7.3](\x, {0.93 /  sqrt(cosh(\x - 0.648 ))});
\end{tikzpicture}
\hfill
\begin{tikzpicture}[xscale= 0.47,yscale=1]
    \draw[step=2,thin,->] (-7.5,0) -- (7.5,0);
    \draw[step=2,thin,->] (0,0) -- (0,2);
\node at (0,-0.5) [] {$0$};
\node at (7.3,-0.5) [] {$x$};
\node at (0.5,1.8) [] {$y$};
\draw[color=black] 
plot[samples=50,domain=-7.4:-5.3](\x, {0.93 / sqrt(cosh( -5.3 - 3.9124 )))});
\draw[color=black] 
plot[samples=50,domain=-5.3:0](\x, {0.93 / sqrt(cosh( \x - 3.9124 )))});
\draw[color=black] plot[samples=50,domain=0:7.3](\x, {0.93 /  sqrt(cosh(\x - 0.0016 ))});
\end{tikzpicture}
\caption{Graph of the Ground State for $\tau=1.2$ on the left and $\tau=5$ on the right.}
\label{fig:GS}
\end{figure}

\begin{figure}[h]
\begin{tikzpicture}[xscale= 0.47,yscale=1]
     \draw[step=2,thin,->] (-7.5,0) -- (7.5,0);
    \draw[step=2,thin,->] (0,0) -- (0,2);
\node at (0,-0.5) [] {$0$};
\node at (7.3,-0.5) [] {$x$};
\node at (0.5,1.8) [] {$y$};
\draw[color=black] plot[samples=50,domain=-7.4:0](\x, {0.93 / sqrt(cosh( \x + 1.161 ))});
\draw[color=black] plot[samples=50,domain=0:7.3](\x, {0.93 /  sqrt(cosh(\x + 0.648 ))});
\end{tikzpicture}
\hfill
\begin{tikzpicture}[xscale= 0.47,yscale=1]
    \draw[step=2,thin,->] (-7.5,0) -- (7.5,0);
    \draw[step=2,thin,->] (0,0) -- (0,2);
\node at (0,-0.5) [] {$0$};
\node at (7.3,-0.5) [] {$x$};
\node at (0.5,1.8) [] {$y$};
\draw[color=black] plot[samples=50,domain=-7.4:0](\x, {0.93 / sqrt(cosh( \x + 3.9124 ))});
\draw[color=black] plot[samples=50,domain=0:7.3](\x, {0.93 /  sqrt(cosh(\x + 0.0016 ))});
\end{tikzpicture}
\caption{Graph of the other stationary state for $\tau=1.2$ on the left and $\tau=5$ on the right.}
\label{fig:stat-states}
\end{figure}
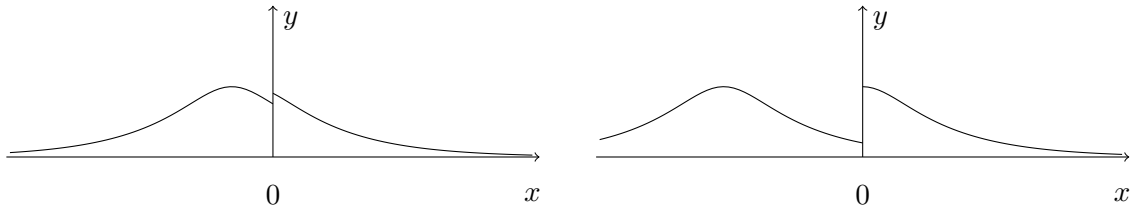

From the technical point of view, the most delicate point in the proof is the existence of Ground States at the critical mass. In the standard NLS case (see for example \cite{W-82}), the authors take advantage of the radial nature of the problem in order to obtain the compactness properties of the minimizing sequences and, as a consequence, the existence of Ground States. In our case, this is not possible since the jump at the origin breaks the radial symmetry of the problem and makes the recovery of compactness more delicate. In order to do it, we use an ad hoc rearrangement procedure (see Lemma \ref{lem:rearr}) combined with an estimate of the Gagliardo-Nirenberg optimal constant for sequences $u_n\deb 0$ weakly converging in $H^1(\Rm)\oplus H^1(\Rp)$ (see Lemma \ref{lem:gn-class-const}).

We state now the Theorem about the $L^2$ critical case in presence of an attractive \emph{F\"{u}l\"{o}p-Tsutsui $\delta$} condition at the origin.

\begin{theorem}[Stationary states for $\sigma=2$ and $\alpha>0$]
\label{thm:gs-crit}
Let $\sigma=2$, $\alpha>0$, $\tau>1$ and $\mu^\star$ and $\widetilde{\mu}$ be as in \eqref{eq:mu-star-exact} and \eqref{eq:mu-tilde}. Then 
\begin{enumerate}
    \item the infimum of the constrained energy is given by
    \begin{equation}
    \E_\alpha(\mu)=
    \begin{cases}
        -c\quad &\text{if}\quad 0<\mu<\mu^\star,\\
        -\infty &\text{if}\quad \mu\geq\mu^\star,
    \end{cases}
\end{equation}
with $c>0$ depending on $\alpha,\tau,\mu$.
\item Ground states of $E_\alpha$ at mass $\mu$ exist if and only if $0<\mu<\mu^\star$.
\item Further stationary states of $E_\alpha$ exist if and only if $\frac{\sqrt{3}}{2}\pi<\mu<\widetilde{\mu}$. 
\end{enumerate} 
\end{theorem}

As one can see, the presence of the additional attractive $\delta$ interaction does not change the value of the critical mass $\mu^\star$ with respect to Theorem \ref{thm:gs-crit-dip}. On the other hand, below the critical mass the level becomes negative and we gain existence of Ground States, while at the critical mass Ground States fail to exist. A similar behaviour happens for the further family of stationary states. In particular, while in absence of $\delta$ interaction they all have mass equal to $\widetilde{\mu}$ (see Theorem \ref{thm:gs-crit-dip}), when $\alpha>0$ the second family of stationary states covers all the masses between $\frac{\sqrt{3}}{2}\pi$ and $\widetilde{\mu}$, where the endpoints are excluded. Let us highlight that in the interval $\left[\mu^\star,\frac{\sqrt{3}}{2}\pi\right]$ no stationary solutions exist.

\section{Auxiliary results}
\label{sec:aux-res}

In the following we use the notation
\begin{equation}
\label{eq:Dtaumu}
\Dtaumu:=\Dtau\cap\{\|u\|_{L^2(\R)}^2=\mu\},
\end{equation} 
and the 
abbreviations
$\|u'\|^2_2$ and $\|u\|^p_p$ to denote $\|u'\|^2_{L^2(\R^-)}+\|u'\|^2_{L^2(\R^+)}$ and $\|u\|^p_{L^p(\R)}$ respectively. 

This section contains two results that will be useful in the following.

The first one concerns the
Gagliardo-Nirenberg inequalities for function in $\Dtau$. Let us first recall the Gagliardo-Nirenberg inequalities on $\R$. For every $\sigma>0$
\begin{equation}
\label{gnp}
\|u\|_{L^{{2\sigma+2}}(\R)}^{{2\sigma+2}}\leq C\|u'\|_{L^{2}(\R)}^{\sigma}\|u\|_{L^2(\R)}^{\sigma+2}\quad \forall\, u\in H^1(\R),
\end{equation}
and
\begin{equation}
\label{gninf}
\|u\|_{L^\infty(\R)}^{2}\leq \|u'\|_{L^2(\R)}\|u\|_{L^2(\R)}\quad \forall \, u\in H^1(\R).
\end{equation}

Analogous inequalities hold also for functions belonging to $\Dtau$, as the next lemma shows.

\begin{lemma}
\label{lem:gn-dtau}
For every $\sigma\in [0,+\infty]$, there exists $C=C(\sigma)>0$ such that, if $\sigma>0$, then
\begin{equation}
\label{gnp-dtau}
\|u\|_{2\sigma+2}^{2\sigma+2}\leq C \|u'\|_{2}^{\sigma}\|u\|_{2}^{\sigma+2},\quad\forall\, u\in \Dtau,
\end{equation}
and if $\sigma=+\infty$, then
\begin{equation}
\label{gninf-dtau}
\|u\|_{\infty}^{2}\leq C \|u'\|_{2}\|u\|_{2},\quad\forall\,u\in\Dtau.
\end{equation}
\end{lemma}
\begin{proof}
 Let $u\in \Dtau$ and, denoted with $u_\pm:\R\to\R$ the functions defined as $u_\pm(x):=u(\pm|x|)$, it follows that $u_\pm\in H^1(\R)$, they are even functions and $u=\chi_- u_- +\chi_+ u_+$. As a consequence,  there results that 
 \begin{equation*}
     \|u_-\|_{2\sigma+2}^{2\sigma+2}+\|u_+\|_{2\sigma+2}^{2\sigma+2}=2\|u\|_{2\sigma+2}^{2\sigma+2},\quad \forall \,\sigma\geq 0,
     %\quad\text{and}\quad \|u'_-\|_2^2+\|u'_+\|_2^2=2\|u'\|_2^2,
 \end{equation*}
hence $\|u_\pm\|_{2\sigma+2}^{2\sigma+2}\leq 2\|u\|_{2\sigma+2}^{2\sigma+2}$ for every $\sigma\geq 0$ and $\|u'_\pm\|_2^2\leq 2\|u'\|_2^2$.
Therefore, by using \eqref{gnp} one gets
\begin{equation*}
\begin{split}
\|u\|_{2\sigma+2}^{2\sigma+2} & \leq \frac{1}{2}\left(\|u_-\|_{2\sigma+2}^{2\sigma+2}+\|u_+\|_{2\sigma+2}^{2\sigma+2}\right)\\
&\leq C\left(\|u_-'\|_{2}^{\sigma}\|u_-\|_{2}^{\sigma+2}+\|u_+'\|_{2}^{\sigma}\|u_+\|_{2}^{\sigma+2}\right)\\
& \leq C \|u'\|_{2}^{\sigma}\|u\|_{2}^{\sigma+2}.
\end{split}
\end{equation*}
On the other hand, when $\sigma=+\infty$, one can observe that
\begin{equation*}
\|u\|_\infty^2=\max\{\|u_-\|_{L^\infty(\R)}^2,\|u_+\|_{L^\infty(\R)}^2\}
\end{equation*}
 and proceeds similarly as for the case $\sigma>0$, using \eqref{gninf} instead of \eqref{gnp}.
 \end{proof}

 \begin{remark}
     Lemma \ref{lem:gn-dtau} provides Gagliardo-Nirenberg inequalities in $\Dtau$ without focusing for the moment on the exact values of the optimal constants. As we will see, the exact value of the optimal constant will play a crucial role in the $L^2$ critical case, in particular in the proof of Theorem \ref{thm:gs-crit-dip}. 
 \end{remark}

The second result consists of a rearrangement procedure that will be crucial in the following. 

\begin{lemma}
\label{lem:rearr}
Let $\tau>1$, $\alpha\in \R$, $u\in \Dtaumu$, $u\geq 0$. Then there exists $u^\star\in \Dtaumu$, $u^\star\geq 0$ such that 
\begin{itemize}
    \item [$(i)$] $E_\alpha(u^\star)\leq E_\alpha(u)$,
\item [$(ii)$] $u^\star$ is monotonically increasing on $\R^-$ and there exists $x_M\in [0,+\infty)$ such that $u^\star$ is monotonically increasing on $[0,x_M]$ and monotonically decreasing on $[x_M,+\infty)$.
\end{itemize} 
Moreover, the equality in $(i)$ holds if and only if the cardinality of the set $\{u=t\}$ is equal to $1$ for almost every $t\in[u(0^-),u(0^+)]$ and equal to $2$ for almost every $t\in (0,u(0^-))\cup(u(0^+),\max u]$.
\end{lemma}
\begin{proof}
    Let us distinguish between the cases $u(0^-)=0$ and $u(0^-)>0$. We start by assuming  $u(0^-)=0$ and consequently
    $u (0^+) = \tau u (0^-) = 0$. Suppose that $\|u\|_{L^\infty(\Rp)}\geq \|u\|_{L^\infty(\Rm)}$ and denote by $x_M>0$ a maximum point for $u$, whose existence is assured by continuity. Then, consider the following construction. Let $\overline{u}\in H^1(0,x_M)$ be the monotone rearrangement of %the function $u_{|(0,x_M)}$, that is 
    the restriction of $u$ to the interval $(0,x_M)$. By definition of monotone rearrangement and the P\'olya-Szeg\H{o} inequality, we get $\overline{u}(0)=\|u\|_\infty$, $\overline{u}(x_M)=0$ and
    \begin{equation}
    \label{eq:PS-bar}\|\overline{u}'\|_{L^2(0,x_M)}\leq \|u'\|_{L^2(x_M)},\quad \|\overline{u}\|_{L^p(0,x_M)}=\|u\|_{L^p(0,x_M)},\quad p\geq 1.
    \end{equation}
Let  $\widetilde{u}\in H^1(\Rp)$ be the monotone rearrangement of the restriction of $u$ to $\R\setminus(0,x_M)$, so that $\widetilde{u}(0)=\|u\|_\infty$ and 
\begin{equation}
    \label{eq:PS-tilde}\|\widetilde{u}'\|_{L^2(\Rp)}\leq \|u'\|_{L^2(\R\setminus(0,x_M))},\quad \|\widetilde{u}\|_{L^p(\Rp)}=\|u\|_{L^p(\R\setminus(0,x_M))},\quad p\geq 1.
    \end{equation}
    Therefore, the function $u^\star$, defined as
    \begin{equation*}
        u^\star(x):=
        \begin{cases}
            0,\quad &x\in\Rm,\\
            \overline{u}(x_M-x), &x\in (0,x_M)\\
            \widetilde{u}(x-x_M), &x\in [x_M,+\infty),
        \end{cases}
    \end{equation*}
    belongs to $\Dtaumu$ since $u^*(0^+) =
    \tau u^* (0^-) = 0$ and satisfies $(i)$ and $(ii)$. If
    $\|u\|_{L^\infty(\Rp)} < \|u\|_{L^\infty(\Rm)}$, then we consider the function $v (x) = u (-x)$ and repeat for $v$ the construction 
    carried out for $u$. This
    concludes the proof for the case $u(0^-)=0$.
    
    Let us now suppose that $u(0^-)>0$ and define the set $S:=\{x\in \R\setminus\{0\}\,:\,u(x)>u(0^+)\}$. Note that  $u(S)=(u(0^+),\|u\|_\infty]$ is connected and every $t\in u(S)$ is attained at least twice on $\R\setminus\{0\}$, except possibly $\|u\|_\infty$. Therefore, setting $L = |S| / 2$ and denoting by $\widehat{u}\in H^1(-L,L)$ the symmetric rearrangement on the interval $(-L,L)$ of the restriction of $u$ on $S$, one gets
    \begin{equation}
\label{eq:PS-hat}\|\widehat{u}'\|_{L^2(-L,L)}\leq \|u'\|_{L^2(S)},\quad \|\widehat{u}\|_{L^p(-L,L)}=\|u\|_{L^p(S)},\quad \forall\,p\geq 1.  
    \end{equation}
Let us then define the set $I:=\{x\in \R\setminus\{0\}\,:\,u(0^-)<u(x)\leq u(0^+)\}$ and note that $u(I)\subset[u(0^-),u(0^+)]$. Hence, denoting by $u^\dag\in H^1([0,D))$ the monotone rearrangement of the function $u_{|I}$ on the interval $[0,D)$, with $D=|I|$, we have 
\begin{equation}
\label{eq:PS-dag}
\|(u^\dag)'\|_{L^2([0,D))}\leq \|u'\|_{L^2(I)},\quad \|u^\dag\|_{L^p([0,D)])}\leq \|u\|_{L^p(I)},\quad \forall\,p\geq 1.
\end{equation}
Finally, define $M:=\R\setminus(J\cup I\cup\{0\})$ and note that $u(M)\subset[0,u(0^-)]$ is connected and every $t\in u(M)$ is attained at least twice on $\R\setminus\{0\}$. Therefore, denoting by $\mathring{u}\in H^1(\R)$ the symmetric rearrangement on $\R$ of the restriction of  $u$ on $M$, we have 
\begin{equation}
\label{eq:PS-ring}
\|(\mathring{u})'\|_{L^2([0,D))}\leq \|u'\|_{L^2(I)},\quad \|\mathring{u}\|_{L^p([0,D))}\leq \|u\|_{L^p(I)},\quad \forall\,p\geq 1.
\end{equation}
The function $u^\star$ defined as
\begin{equation*}
    u^\star(x):=
    \begin{cases}
    \mathring{u}(x),\quad &x\in \Rm,\\
    \widehat{u}(x-L),\quad &x\in (0,2L),\\
    u^\dag(x-2L),\quad &x\in [2L,2L+D),\\
    \mathring{u}(x-2L-D),\quad &x\in[2L+D,+\infty)
    \end{cases}
\end{equation*}
belongs to $\Dtaumu$ since 
$$
u^* (0^+) = \widehat u (-L) = u (0^+) =
\tau u (0^-) = \tau \mathring u (0^-) =
\tau u^* (0^-)
$$
and satisfies $(i)$ and $(ii)$.

It is left to prove that the equality in $(i)$ is realized if and only if the cardinality of the set $\{u=t\}$ is equal to $1$ for almost every $t\in[u(0^+),u(0^+)]$ and equal to $2$ for almost every $t\in (0,u(0^-))\cup(u(0^+),\max u]$. In order to prove it, it is sufficient to specify in which cases the equality  occurs in each P\'olya-Szeg\H{o} inequality used along the proof, i.e. in \eqref{eq:PS-bar}, \eqref{eq:PS-tilde}, \eqref{eq:PS-hat}, \eqref{eq:PS-dag} and \eqref{eq:PS-ring}. So, for the monotone rearrangements  in \eqref{eq:PS-bar}, \eqref{eq:PS-tilde} and \eqref{eq:PS-dag}, the equality is realized if and only if the number of preimages equals a.e. $1$.  On  the other hand, for the symmetric rearrangements in \eqref{eq:PS-hat} and \eqref{eq:PS-ring},  equality occurs if and only if the number of preimages equals a.e. $2$.
By this considerations, the proof is complete.
\end{proof}

The last result we need is classical 
(\cite{C-03,ZS-72,GSS-87}).

\begin{theorem} \label{thm:nls}
For every $\mu > 0$ there is a unique positive minimizer $\varphi_\mu$ at mass $\mu$
for  the standard NLS energy functional
\begin{equation} \label{enls}
E_{\NLS}(u,\R)=\frac{1}{2}\| u'\|_{L^2(\R)}-\frac{1}{2\sigma+2}\| u \|_{L^{2\sigma+2}(\R)}^{2\sigma+2}.
\end{equation}
 Its explicit expression is
\begin{equation}
%    \label{soliton}
   \varphi_\mu(x) = C_\sigma\mu^{\f{1}{2-\sigma}} \sech^{\f{1}{\sigma}}\left(c_\sigma\mu^{\f{\sigma}{2-\sigma}} x\right)\,,
\end{equation}
where $C_\sigma,\,c_\sigma>0$ depends on $\sigma$ only.

All other minimizers of \eqref{enls} coincide with $\varphi_\mu$ up to a constant
phase.
\end{theorem}

\section{Proof of Theorem \ref{thm:gs-sub}: the  subcritical case}
\label{sec:subcritical}

\begin{comment}
Given $X=\R$, $X=\Rp$ or $X=\Rm$, let us define the energy level 
\begin{equation} 
\widetilde{\E}(\mu,X):=\inf_{\substack{u\in H^1_\mu(X)\\u(0)=0}}\left\{\frac{1}{2}\|u'\|_{L^2(X)}^2-\frac{1}{{2\sigma+2}}\|u\|_{L^{2\sigma+2}(X)}^{2\sigma+2}\right\}.
\end{equation}
In the following, we will write simply $\widetilde{\E}(\mu)$ when $X$ coincides with $\R$.
\end{comment}
This section is devoted to the proof of Theorem \ref{thm:gs-sub}. 

\begin{proof}[Proof of Theorem \ref{thm:gs-sub}]

Let us divide the proof into two steps.\\

\begin{comment}
\textit{Step 1. $\widetilde{\E}(\mu)=E_{NLS}(\varphi_\mu)$.} Indeed, let us first observe that
\begin{equation}
\label{eq:einf-sol}
  \widetilde{\E}(\mu,\R^\pm)=E_{NLS}(\varphi_\mu).  
\end{equation}
Since by straightforward arguments there results that
\begin{equation*}
    \widetilde{\E}(\mu)=\inf\left\{\widetilde{\E}(\nu,\Rp)+\widetilde{\E}(\mu-\nu,\Rm)\,:\, 0\leq\nu\leq\mu\right\},
\end{equation*}
by using \eqref{eq:einf-sol} and the fact that $E_{NLS}(\varphi_M)=-\theta_\sigma M^\frac{\sigma+2}{2-\sigma}$ is a strictly concave function of $M\geq 0$, we deduce that $\widetilde{\E}(\mu)=E_{NLS}(\varphi_\mu)$.
\end{comment}

\textit{Step 1.}
Here we prove $\E_\alpha(\mu)<E_{\NLS}(\varphi_\mu)$, where $\E_\alpha (\mu)$ was defined in \eqref{ealfa}. In order to prove this inequality, we exhibit a function $v\in \Dtaumu$ such that $E_{\alpha}(v)<E_{\NLS}(\varphi_{\mu})=-\theta_{p}\mu^{\frac{\sigma+2}{2-\sigma}}$. Specifically, we claim that there exists $\nu\in(0\,,\mu)$ such that the function 
\begin{equation*}
v:=\begin{cases}
\varphi_{2\nu} \quad&\text{on}\quad \Rm,\\
\varphi_{2(\mu-\nu)}\quad &\text{on}\quad \Rp,
\end{cases}
\end{equation*}
belongs to $\Dtaumu$ and satisfies $E_{\NLS}(v)<E_{\NLS}(\varphi_{\mu})$. In particular, taking advantage of the definition of $v$ and recalling that $\varphi_{M}(0)=\varphi_1(0)M^{\frac{1}{2-\sigma}}$ for every $M>0$, there results that $\nu$ must satisfy the following system
\begin{equation}
\label{eq:system-nu}
    \begin{cases}
        \nu^{\frac{1}{2-\sigma}}=\tau (\mu-\nu)^{\frac{1}{2-\sigma}},\\
        \nu^{\frac{\sigma+2}{2-\sigma}}+(\mu-\nu)^{\frac{\sigma+2}{2-\sigma}}>2^{-\frac{2\sigma}{2-\sigma}}\mu^{\frac{\sigma+2}{2-\sigma}}.
    \end{cases}
\end{equation}
One immediately sees that the first equation in \eqref{eq:system-nu} has a unique solution given by $\nu=\frac{\tau^{2-\sigma}}{1+\tau^{2-\sigma}}\mu$, thus one has to check only that such solution satisfies also the inequality in \eqref{eq:system-nu}. By straightforward computation, this reduces to the inequality 
\begin{equation*}
    \left(\frac{1+\tau^{2-\sigma}}{2}\right)^\frac{\sigma+2}{2-\sigma}<\frac{1+\tau^\frac{\sigma+2}{2-\sigma}}{2},
\end{equation*}
which is true since the function $(\cdot)^\frac{\sigma+2}{2-\sigma}$ is strictly convex,
so the proof of Step 2 is complete.\\

\textit{Step 2. Existence of Ground States.}  Let $u_{n}$ be a minimizing sequence for $E_\alpha$, i.e. $u_n \in D^\tau, \ \|u_{n}\|_{2}^{2}=\mu$ and $E_\alpha(u_{n})\to \E_\alpha(\mu)$. By applying \eqref{gnp} and \eqref{gninf} to \eqref{eq:energy}, one gets
\begin{equation*}
    E_\alpha(u_n)\geq \frac{1}{2}\|u_n'\|_2^2-C\|u_n'\|_2^\sigma\mu^{\frac{\sigma}{2}+1}-C\|u_n'\|_2\|\mu^\frac{1}{2}.
\end{equation*}
Since $E_\alpha(u_n)$ 
is a bounded sequence
and $\sigma<2$, then $u_n$ is a bounded sequence in $D^\tau$. This entails that $\|u_n\|_{H^1(\Rm)}^2+\|u_n\|_{H^1(\Rp)}^2$ is  bounded, thus, by Banach-Alaoglu's theorem, there exists $u$ such that, up to subsequences, $u_{n}\deb u$ weakly in $H^1(\Rm)\oplus H^1(\Rp)$. Since $u_{n}\to u$ in  $L^\infty_{\text{loc}}(\R\setminus\{0\})$ and thus $u_n(0^\pm)\to u(0^\pm)$, it follows that $u\in\Dtau$. Let us now denote  $m:=\|u\|_2^2$ and observe that $m\leq \mu$ by weak lower semicontinuity of the $L^2$ norm.

Suppose first that $m=0$. In this case, $u\equiv 0$ and so $u_n(0^-)\rightarrow 0$ as $n \to +\infty$. Let us define for every $n\in\mathbb{N}$ the function $w_n:=\frac{\sqrt{\mu}}{\|v_n\|_2}v_n$, where
\begin{equation*}
v_n(x):=
\begin{cases}
    u_n(x+|u_n(0^-)|),\quad &x<-|u_n(0^-)|,\\
    -\frac{u_n(0^-)}{|u_n(0^-)|}x,\quad &-|u_n(0^-)|\leq x\leq 0,\\
    \frac{u_n(0^+)}{|u_n(0^+)|}x,\quad &0<x\leq |u_n(0^+)|,\\
    u_n(x-|u_n(0^+)|),\quad &x>|u_n(0^+)|.\\
    \end{cases}
\end{equation*}
It is easy to check that $w_n\in H^1_\mu(\R)$ and $E_\alpha(w_n)-E_\alpha(u_n)\to 0$ as $n\to+\infty$. Therefore, by Step 1  there results that
\begin{equation*}
E_{\NLS}(\varphi_\mu)>\E_\alpha(\mu)=\lim_{n\to+\infty}E_\alpha(u_n)=\lim_{n\to+\infty} E_\alpha(w_n)\geq E_{\NLS}(\varphi_\mu),
\end{equation*}
which is a contradiction, hence $m>0$. 

Suppose now that $0<m<\mu$. On the one hand, since $\frac{\mu}{\|u_n-u\|_2^2}\to\frac{\mu}{\mu-m}>1$ as $n\to+\infty$ by weak convergence in $L^2(\R)$ and $\sigma>0$, then 
\begin{equation*}
    \E_\alpha(\mu)\leq E_\alpha\left(\sqrt{\frac{\mu}{\|u_n-u\|_2^2}}(u_n-u)\right)<\frac{\mu}{\|u_n-u\|_2^2}E_\alpha(u_n-u),\quad \text{for}\quad n\gg 1,
\end{equation*}
leading to
\begin{equation}
\label{eq:intmas1}
    \liminf_n E_\alpha(u_n-u)\geq \frac{\mu-m}{\mu}\E_\alpha(\mu).
\end{equation}
On the other hand, in a similar way one obtains 
\begin{equation}
\label{eq:intmas2}
E_\alpha(u)>\frac{m}{\mu}\E_\alpha(\mu).    
\end{equation}
Moreover, by exploiting the convergences of $u_n$ to $u$ weakly in $H^1(\Rm)\oplus H^1(\Rm)$ and a.e. on $\R$  and using Brezis-Lieb lemma \cite{bl}, there results
\begin{equation}
\label{eq:intmas3}
    E_\alpha(u_n)=E_\alpha(u_n-u)+E_\alpha(u)+o(1), \quad\text{as}\quad n\to+\infty.
\end{equation}
Combining \eqref{eq:intmas1}, \eqref{eq:intmas2} and \eqref{eq:intmas3}, we get
\begin{equation*}
    \E_\alpha(\mu)=\lim_n E_\alpha(u_n)=\lim_n E_\alpha(u_n-u)+E_\alpha(u)>\E_\alpha(\mu),
\end{equation*}
which is a contradiction, hence $m=\mu$ and $u\in \Dtaumu$. In particular, $u_n\to u$ in $L^{2\sigma+2}(\R)$ since $u_n\to u$ in $L^2(\R)$ and $(u_n)_n$ is bounded in $L^\infty(\R)$, thus by weak lower semicontinuity
\begin{equation*}
E_\alpha(u)\leq\liminf_n E_\alpha(u_n)=\E_\alpha(\mu),
\end{equation*}
namely $u$ is a Ground State of $E_\alpha$ at mass $\mu$. The fact that the Ground State is unique and the existence of a further positive stationary solution for $\mu>\mu_\alpha$  follow by Proposition \ref{prop:stat-sub}.
\end{proof}

\begin{comment}
The next lemma states a well-known relation between ground states at fixed mass and minimizers of the action at fixed frequency \cite. 

\begin{lemma}
If $u$ is a ground state of $E_\alpha$ at mass $\mu$, then  $u$ is a minimizer of $S_\omega$ at frequency 
\begin{equation}
\label{eq:omega}
    \omega=\mu^{-1}\left(\|u\|_{2\sigma+2}^{2\sigma+2}-\|u'\|_2^2+\alpha |u(0^-)|^2\right).
\end{equation}
\end{lemma}
\begin{proof}
Let $u$ be a ground state of $E_\alpha$ at mass $\mu$ and $\omega$ given by \eqref{eq:omega} be the associated Lagrange multiplier. Suppose by contradiction that there exists $v\in N_\omega$ such that $S_\omega(v)<S_\omega(u)$. Let now $\la>0$ be such that $\|\la v\|_2^2=\mu$. Then
\begin{equation*}
    S_\omega(\la v)=\frac{\la^2}{2}\left(\|u'\|_2^2-\alpha|u(0^-)|^2+\om\|v\|_2^2\right)-\frac{\la^{2\sigma+2}}{2\sigma+2}\|v\|_{2\sigma+2}^{2\sigma+2},
\end{equation*}
so that, by using that $v\in N_\om$, we get
\begin{equation*}
    \frac{d}{d\la}S_\om(\la v)=\la \left(\|u'\|_2^2-\alpha|u(0^-)|^2+\om\|v\|_2^2\right)-\la^{2\sigma+1}\|v\|_{2\sigma+2}^{2\sigma+2}=\la(1-\la^{2\sigma})\|v\|_{2\sigma+2}^{2\sigma+2},
\end{equation*}
which is greater or equal to zero if and only if $0<\la\leq 1$. Therefore, $S_\omega(\la v)\leq S_\omega(v)<S_\omega(u)$, which, together with $\|\la v\|_2^2=\|u\|_2^2$, entails that $E_\alpha(\la v)<E_\alpha(u)$, getting a contradiction and concluding the proof.
\end{proof}
\end{comment}

\begin{remark}
\label{rem:resonance}

The value $\mu_\alpha$ in the statement of Theorem \ref{thm:gs-sub} corresponds to the mass of a soliton of frequency $\omega_\alpha=\frac{\alpha^2}{(\tau^2-1)^2}$.
Such frequency corresponds to a family of solutions to the system
\begin{equation*}
\begin{cases}
    u''=\omega_\alpha u\quad \text{on}\quad \R\setminus\{0\},\\
    u(0^+)=\tau u(0^-),\\
    u'(0^-)-\tau u'(0^+)=\alpha u(0^-),
    \end{cases}
\end{equation*}
given by the multiples of the function 
\begin{equation*}
    v(x)=\begin{cases}
        \tau e^{-\sqrt{\omega_{\alpha}}x},\quad x\in \Rp,\\
        e^{-\sqrt{\omega_{\alpha}}x},\quad x\in \Rm.
    \end{cases}
\end{equation*}
Such functions %formally solve the eigenvalue equation $H_{\tau,\alpha}v=-\omega_\alpha v$, %solve the but do not belong to $L^2(\R)$: 
solve the stationary equation but are not in $L^2(\R)$, so they are resonances.

Let us highlight that $\omega_\alpha$ is larger than the opposite of the first eigenvalue of the linear operator, namely 
\begin{equation*}
	\omega_\alpha>\frac{\alpha^2}{(\tau^2+1)^2}=-\inf\left\{\|u'\|_{L^2(\Rm)}^2+\|u'\|_{L^2(\Rp)}^2-\alpha|u(0^-)|^2\,:\, u\in \Dtau\,,\,\|u\|_{L^2(\R)}^2=1\right\}. 
\end{equation*}
In order to check this fact, it is sufficient to look for solutions in $\Dtau$ of the system
\begin{equation}
	\label{eq:lin-sys}
\begin{cases}
	u''=\omega u\quad \text{on}\quad \R\setminus\{0\},\\
	u(0^+)=\tau u(0^-),\\
	u'(0^-)-\tau u'(0^+)=\alpha u(0^-),
\end{cases}
\end{equation}
that must be of the general form
\begin{equation}
	\label{eq:eigenfunct}
	u(x)=\begin{cases}
		c\tau e^{-\sqrt{\omega}x}, \quad &x>0\\
		c e^{\sqrt{\omega}x}, &x<0.
	\end{cases}
\end{equation} 
By direct computation, the only $\omega>0$ for which \eqref{eq:lin-sys} admits a solution of the form \eqref{eq:eigenfunct} is $\omega=\frac{\alpha^2}{(\tau^2+1)^2}$. 

\begin{comment}
At a qualitative level, one can expect that 
the frequency of a nonlinear Ground State is larger than the frequency of the
linear one, since the additional nonlinear term is attractive and then enhances the
absolute value of the energy of the bound.
\end{comment}
%They represent resonances and will not be studied in the present paper.
\end{remark}

\section{Proof of Theorems \ref{thm:gs-crit-dip} and \ref{thm:gs-crit}: the  critical case}

\label{sec:critical}

This section contains the proof of Theorems \ref{thm:gs-crit-dip} and \ref{thm:gs-crit}, that deal with the $L^2$-critical case first in presence of a dipole interaction and then  of a  F\"{u}l\"{o}p-Tsutsui interaction.

Let us first focus on the model with dipole interaction (i.e. $\alpha=0$ in \eqref{eq:energy}). We preliminary observe that 
\begin{equation}
\label{eq:E0leq0}
\E_0(\mu)\leq 0\quad \forall\,\mu>0
\end{equation}
and 
\begin{equation}
\label{eq:E0-inf}
    \E_0(\mu)=-\infty\quad\text{if and only if}\quad \exists\, u\in \Dtaumu\,:\,E_0(u)<0.
\end{equation}
Relations \eqref{eq:E0leq0} and \eqref{eq:E0-inf} follows straightforwardly since, given $u\in \Dtaumu$ and denoted by $u_\la(x):=\sqrt{\la}u(\la x)$, there results that $u_\la\in \Dtaumu$ and
\begin{equation*}
E_0(u_\la)=\la^2 E_0(u),
\end{equation*}
so that $E_0(u_\la)\to 0$ as $\la\to 0^+$ if $E_0(u)>0$ and $E_0(u_\la)\to-\infty$ as $\la \to +\infty$ if $E_0(u)<0$.

Before proving Theorem \ref{thm:gs-crit-dip}, we need to prove two lemmas. The former is about the optimal constant in Gagliardo-Nirenberg inequality \eqref{eq:gn-crit} for sequences $u_n$ weakly convergent to $0$ in $H^1(\Rm)\oplus H^1(\Rp)$.

\begin{lemma}
\label{lem:gn-class-const}
Let $v_n\in \Dtau$ be a sequence such that $v_n\deb 0$ in $H^1(\Rm)\oplus H^1(\Rp)$. Then 
\begin{equation}
    \|v_n\|_6^6\leq \frac{4}{\pi^2}\|v_n\|_2^4\|v_n'\|_2^2+o(1), \quad{\rm{as}}\quad n \to+\infty.
\end{equation}
\end{lemma}
\begin{proof}
Let $v_n$ be as in the statement of the lemma and define for every $n$ the function $w_n$ as
\begin{equation*}
w_n(x):=\begin{cases}
    v_n(x)-|v_n(0^+)|\quad &\text{if}\quad v_n(x)>|v_n(0^+)|,\\
    v_n(x)+|v_n(0^+)|\quad &\text{if}\quad v_n(x)<-|v_n(0^+)|,\\
    0\quad &\text{elsewhere}.
\end{cases}.
\end{equation*}
Since $v_n\deb 0$ in $H^1(\Rm)\oplus H^1(\Rp)$, we have $|v_n(0^+)|\to 0$. Moreover, since $v_n$ is a bounded sequence in $L^2(\R)$, there results 
\begin{equation}
\label{eq:conv-w-v}
\|v_n-w_n\|_6^6\leq \|v_n-w_n\|_{L^\infty(\R)}^4\|v_n-w_n\|_{L^2(\R)}^2\leq 
C |v_n(0^+)|^4 \|v_n\|_2^2\to 0, \quad n\to+\infty.
\end{equation}
Now, since $w_n\in H^1(\R)$, it must satisfy the Gagliardo-Nirenberg inequality
\begin{equation}
\label{eq:gn-class}
\|w_n\|_{L^6(\R)}^6\leq \frac{4}{\pi^2}\|w_n\|_{L^2(\R)}^4\|w_n'\|_{L^2(\R)}^2.
\end{equation}
Since $\|w_n'\|_{L^2(\R)}^2\leq\|v_n'\|_{L^2(\Rm)}^2+\|v_n'\|_{L^2(\Rp)}^2$ and $\|w_n\|_{L^2(\R)}^2\leq\|v_n\|_{L^2(\R)}^2$, by applying \eqref{eq:conv-w-v} to \eqref{eq:gn-class} we get the thesis.
\end{proof}

Let us observe now that Ground States at fixed mass $\mu>0$ solve for some $\omega>0$ the system
\begin{equation}
\label{eq:stat-sol-dip}
\begin{cases}
    u''+|u|^{4}u=\om u\quad \text{on}\quad \R\setminus\{0\},\\
    u(0^+)=\tau u(0^-),\\
    u'(0^-)=\tau u'(0^+).
\end{cases}
    \end{equation}
    
In \cite[Propositions 8.8 and 8.9]{anv}, the authors classified all the real solutions to \eqref{eq:stat-sol-dip}. In the next lemma, we recall the explicit expression of the positive solutions of \eqref{eq:stat-sol-dip}, providing some additional informations. 

\begin{lemma}
\label{lem:stat-sol-dip}
Let $\sigma=2$, $\tau>1$ and $\omega>0$. Then there are exactly two solutions $u_1$ and $u_2$ to \eqref{eq:stat-sol-dip} satisfying the additional condition $u(0^+)>0$. In particular, $u_1$ and $u_2$ are given by
\begin{equation*}
 u_1(x)=\chi_1^\tau(x):=\begin{cases}
     (3\omega)^\frac{1}{4}\mathrm{sech}^\frac{1}{2}(2\sqrt{\om}(x-\xi^+)),\quad x\in \Rp,\\
     (3\omega)^\frac{1}{4}\mathrm{sech}^\frac{1}{2}(2\sqrt{\om}(x-\xi^-)),\quad x\in \Rm,
 \end{cases}   
\end{equation*}
and
\begin{equation*}
 u_2(x)=\chi_2^\tau(x):=\begin{cases}
     (3\omega)^\frac{1}{4}\mathrm{sech}^\frac{1}{2}(2\sqrt{\om}(x+\xi^+)),\quad x\in \Rp,\\
     (3\omega)^\frac{1}{4}\mathrm{sech}^\frac{1}{2}(2\sqrt{\om}(x+\xi^-)),\quad x\in \Rm,
 \end{cases}   
\end{equation*}
where 
\begin{equation*}
    \tanh(2\sqrt{\om}\xi^+)=\frac{1}{\sqrt{1+\tau^4}},\quad \tanh(2\sqrt{\om}\xi^-)=\frac{\tau^2}{\sqrt{1+\tau^4}}
\end{equation*}
Moreover, it holds that 
\begin{equation*}
E_0(u_1)=E_0(u_2)=0
\end{equation*}
and
\begin{equation}
\label{eq:mass-u1-u2}
    \begin{split}
    \|u_1\|_2^2& =\frac{\sqrt{3}}{2}\left(\frac{\pi}{2}+2 \arcsin\left(\frac{1}{\sqrt{1+\tau^4}}\right)\right),\\
    \|u_2\|_2^2& =\frac{\sqrt{3}}{2}\left(\frac{3\pi}{2}-2 \arcsin\left(\frac{1}{\sqrt{1+\tau^4}}\right)\right).
    \end{split}
\end{equation}
\end{lemma}

\begin{proof}
The expressions for $u_1$ and $u_2$ can be found in \cite[Proposition 8.8]{anv}. On the other hand, by \cite[Proposition 8.9]{anv}, we have that
\begin{equation}
\label{eq:uitau2}
    \|u_i\|_2^2=\frac{\sqrt{3}}{2}\left(\int_{-1}^1 \frac{1}{\sqrt{1-s^2}}\,ds+(-1)^i\int_{\frac{1}{\sqrt{1+\tau^4}}}^{\frac{\tau^2}{\sqrt{1+\tau^4}}}\frac{1}{\sqrt{1-s^2}}\,ds\right),\quad i=1,2,
\end{equation}
and
\begin{equation}
\label{eq:uitau6}
    \|u_i\|_6^6=\frac{3\sqrt{3}}{2}\omega\left(\int_{-1}^1 \sqrt{1-s^2}\,ds+(-1)^i\int_{\frac{1}{\sqrt{1+\tau^4}}}^{\frac{\tau^2}{\sqrt{1+\tau^4}}}\sqrt{1-s^2}\,ds\right),\quad i=1,2.
\end{equation}
Since $u_i$ solves \eqref{eq:stat-sol-dip} for $i=1,2$, then $u_i\in N_\om$, namely
\begin{equation}
\label{eq:uitau-nehari}
\|(u_i)'\|_2^2+\om\|u_i\|_2^2=\|u_i\|_6^6.   
\end{equation}
Moreover, it holds that 
\begin{equation*}
    \int\sqrt{1-s^2}\,ds=\frac{1}{2}\left(s\sqrt{1-s^2}+\int\frac{1}{\sqrt{1-s^2}}\,ds\right),
\end{equation*}
thus by applying it to \eqref{eq:uitau6} and recalling \eqref{eq:uitau2}, we get
\begin{equation*}
    \|u_i\|_6^6=\frac{3}{2}\omega\|u_i\|_2^2,
\end{equation*}
that, together with \eqref{eq:uitau-nehari}, leads to $E_0(u_i)=0$ for $i=1,2$. Moreover, by direct computations one deduces \eqref{eq:mass-u1-u2}, completing the proof.
\end{proof}

\begin{comment}
\begin{remark}
    Let us observe that the assumption $\tau>1$ in Lemma \ref{lem:stat-sol-dip} is not restrictive and Lemma \ref{lem:stat-sol-dip} holds also in the other regimes of $\tau$. In particular, if $\tau\in (0,1)$, then the two solutions to \eqref{eq:stat-sol-dip} are given by $u_i(x)=\chi_i^\frac{1}{\tau}(-x)$, $i=1,2$. On the other hand, if $\tau<-1$, then the only two solutions are given by
    \begin{equation*}
        u_i(x)=\begin{cases}
            \chi_i^{|\tau|}(x),\quad &x\in \Rp,\\
            -\chi_i^{|\tau|}(x),\quad &x\in \Rm,
        \end{cases}\quad i=1,2,
    \end{equation*}
    while, if $\tau\in (-1,0)$, then the two solutions are given by
    \begin{equation*}
    u_i(x)=
      \begin{cases}
    \chi_i^{\frac{1}{|\tau|}}(-x),\quad &x\in \Rp,\\
    -\chi_i^{\frac{1}{|\tau|}}(-x),\quad &x\in \Rm,
    \end{cases} \quad i=1,2. 
    \end{equation*}
\end{remark}
\end{comment}

We are now ready to prove Theorem \ref{thm:gs-crit-dip}.

\begin{proof}[Proof of Theorem \ref{thm:gs-crit-dip}]

The proof is divided in four steps.\\

\emph{Step 1. Here we prove that
\begin{equation*}
    \E_0(\mu)=
    \begin{cases}
        0\quad &\text{if}\quad 0<\mu\leq\mu^\star,\\
        -\infty &\text{if}\quad \mu>\mu^\star.
    \end{cases}
\end{equation*}
and that Ground States do not exist if $\mu\neq\mu^\star$.}
Let us observe that, plugging \eqref{gnp-dtau} into the expression of $E_0$ and using \eqref{eq:mu-star}, one has 
\begin{equation*}
    E_0(u)\geq \frac{1}{2}\|u'\|_2^2\left(1-\frac{\mu^2}{(\mu^\star)^2}\right)\quad \forall \,u\in\Dtaumu.
\end{equation*}
Thus $E_0(u)\geq0$ for every $u\in \Dtaumu$ if $0<\mu\leq\mu^\star$, entailing that $\E_0(\mu)=0$ if $0<\mu\leq\mu^\star$. In particular, if $0<\mu<\mu^\star$, then $E_0(u)>0$ for every $u\in \Dtaumu$, entailing the nonexistence of Ground States.

Let us now fix $\mu>\mu^\star$. By definition of $K_\tau$ and \eqref{eq:mu-star}, for every $\ep>0$ there exists $v_\ep\in \Dtaumu$ such that
\begin{equation*}
    \|v_\ep\|_6^6\geq(K_\tau-\ep)\|v_\ep'\|_2^2\|v_\ep\|_2^4,
\end{equation*}
thus
\begin{equation*}
E_0(v_\ep)<\frac{1}{2}\left(1-\frac{\mu^2}{(\mu^\star)^2}+\frac{\ep}{3}\mu^2\right)\|v_\ep'\|_2^2.
\end{equation*} 
Choosing $\ep>0$ sufficiently small, one gets $E_0(v_\ep)<0$, thus $\E_0(\mu)=-\infty$ and Ground States do not exist.\\ 
\begin{comment}
Thus, to get \eqref{eq:E0-thm-crit} it is left to prove that $\E_0(\mu^\star)=0$. Indeed, for any $\mu>0$ every function $u\in \Dtaumu$ can be written as $u=\sqrt{\mu}v$, with $v\in \Dtau_1$, so that
\begin{equation*}
\E_0(\mu)=\inf_{v\in\Dtau_1}f_v(\mu),
\end{equation*}
where $f_v(\mu):=\frac{1}{2}\|v'\|_2^2\mu-\frac{1}{6}\|v\|_6^6\mu^3$. Since $f_v$ is a continuous function of $\mu$ for every $v\in \Dtau_1$, then $\E_0$ is an upper semicontinuous function of $\mu$, hence, denoted by $(\mu_n)_n$ a sequence such that $\mu_n\nearrow \mu^\star$, we have
\begin{equation*}
    \E_0(\mu^\star)\geq \limsup_n \E_0(\mu_n)=0.
\end{equation*}
This, combined with \eqref{eq:E0leq0}, leads to $\E_0(\mu^\star)=0$. 
\end{comment}

\emph{Step 2. Here we prove that $\mu^\star<\frac{\sqrt{3}}{2}\pi$.} By \emph{Step 1}, we have  $E_0(u)>0$ for every $u\in\ \Dtaumu$ if $0<\mu<\mu^\star$. Moreover, by Lemma \ref{lem:stat-sol-dip} there exists a function $u_1\in \Dtau$ satisfying $\|u_1\|_2^2 =\frac{\sqrt{3}}{2}\left(\frac{\pi}{2}+2 \arcsin\left(\frac{1}{\sqrt{1+\tau^4}}\right)\right)$ and $E_0(u_1)=0$, hence 
\[
\mu^\star\leq\|u_1\|_2^2= \frac{\sqrt{3}}{2}\left(\frac{\pi}{2}+2 \arcsin\left(\frac{1}{\sqrt{1+\tau^4}}\right)\right)<\frac{\sqrt{3}}{2}\pi.\\
\]

\emph{Step 3. Existence of Ground States at mass $\mu^\star$.}
Let $u_n \subset \Dtau_{\mu^\star}$ be a maximizing sequence for the Gagliardo-Nirenberg inequality \eqref{eq:gn-crit}, namely a sequence such that 
\begin{equation}
\label{eq:max-seq-gn}
    \frac{\|u_n\|_{6}^6}{\|u_n'\|_2^2}\to K_\tau(\mu^\star)^2=3.
\end{equation}
By Lemma \ref{lem:rearr}
one can assume without loss of generality that $u_n$ is monotonically increasing on $\Rm$. Moreover, by performing mass-preserving transformations, one can assume  $\|u_n'\|_2^2=1$ for every $n$. This entails  $E_0(u_n)\to 0$ as $n\to+\infty$, thus $u_n$ is also a minimizing sequence for $E_0$ at mass  $\mu^\star$. In particular, since both $\|u_{n_{|\Rm}}\|_{H^1(\Rm)}$ and $\|u_{n_{|\Rp}}\|_{H^1(\Rp)}$ are bounded, there exists $u\in H^1(\Rm)\oplus H^1(\Rp)$ such that, up to subsequences, $u_n\deb u$ in $H^1(\Rm)\oplus H^1(\Rp)$. In addition, since both $u_{n_{|\Rm}}\to u_{|\Rm}$ in $L^\infty_{\text{loc}}(\Rm)$ and $u_{n_{|\Rp}}\to u_{|\Rp}$ in $L^\infty_{\text{loc}}(\Rp)$, it follows that $u_n(0^\pm)\to u(0^\pm)$, hence $u\in \Dtau$. Let $m:=\|u\|_{L^2(\R)}$ and observe that $m\leq\mu^\star$ by weak lower semicontinuity of the norm. 

Suppose first that $m=0$, i.e. $u\equiv 0$. Since $u_n\deb 0$ in $H^1(\Rm)\oplus H^1(\Rp)$, by applying Lemma \ref{lem:gn-class-const} and dividing by  $\|u_n'\|_2^2$, we get
\begin{equation}
\label{eq:low-bound-opt}
    \frac{\|u_n\|_{L^6(\R)}^6}{\|u_n'\|_2^2}\leq\frac{4}{\pi^2}(\mu^\star)^2+o(1),\quad n\to+\infty.
\end{equation}
By \emph{Step 2}, one has $\frac{4}{\pi^2}(\mu^\star)^2\leq \frac{4}{\pi^2}\|u_1\|_2^4<3$, hence \eqref{eq:low-bound-opt} is in contradiction with \eqref{eq:max-seq-gn}, ruling out the case $m=0$.

Suppose now that $0<m<\mu^\star$. By a standard application of Brezis-Lieb lemma \cite{bl}, there results that 
\begin{equation}
\label{eq:E0-un-u}
E_0(u_n)=E_0(u_n-u)+E_0(u)+o(1),\quad n\to+\infty.
\end{equation}
Since $u_n-u\deb 0$ in $H^1(\Rm)\oplus H^1(\Rp)$, then by Lemma \ref{lem:gn-class-const}
\begin{equation}
\label{eq:E0-un-u-2}
    E_0(u_n-u)\geq \frac{1}{2}\left(1-\frac{4}{3\pi^2}\|u_n-u\|_2^4\right)+o(1),\quad\text{as}\quad n\to+\infty.
\end{equation}
Therefore, by \eqref{eq:E0-un-u}, \eqref{eq:E0-un-u-2} and the fact that $\|u_n-u\|_2^4\leq (\mu^\star)^2<\frac{3\pi^2}{4}$, one has
\begin{equation*}
    E(u)\leq \liminf_n E(u_n)=0.
\end{equation*}
 Since the function $\sqrt{\frac{\mu^\star}{m}}u$ belongs to $\Dtau_{\mu^\star}$ and
\begin{equation*}
    E_0\left(\sqrt{\frac{\mu^\star}{m}}u\right)=\frac{\mu^\star}{m}\frac{1}{2}\|u'\|_2^2-\left(\frac{\mu^\star}{m}\right)^3\frac{1}{6}\|u_n\|_6^6<\frac{\mu^\star}{m}E_0(u)\leq 0,
\end{equation*}
we get a contradiction with the fact that $\E_0(\mu^\star)=0$. Hence,  $m=\mu^\star$ and 
\begin{equation*}
    \|u_n-u\|_6^6\leq K_\tau\|(u_n-u)'\|_2^2\|u_n-u\|_2^4\to 0\quad\text{as}\quad n\to+\infty,
\end{equation*}
entailing that $u$ is a Ground State at mass $\mu^\star$.\\

\emph{Step 4. Conclusion.} The fact that Ground States at mass $\mu^\star$ are also optimizer of the Gagliardo-Nirenberg \eqref{eq:gn-crit} follows from the fact that in \emph{Step 3} we have shown that
a maximizing sequence for inequality \eqref{eq:gn-crit} at mass $\mu^\star$ is also a minimizing sequence for $E_0$ at the same mass, and also the opposite trivially holds.

On the other hand, the exact value of $\mu^\star$ follows from the fact that Ground States of $E_0$ solve \eqref{eq:stat-sol-dip}, hence $\mu^\star$ will be the mass of the least-mass stationary state: in particular, $\mu^\star=\|u_1\|_2^2$, with $u_1$ as in Lemma \ref{lem:stat-sol-dip}. The exact value of $K_\tau$ is computed starting from the expression of $\mu^\star$ and using \eqref{eq:mu-star}.

Finally, the existence of the further positive stationary solution at mass $\widetilde{\mu}$ follows by the analysis of stationary states in Lemma \ref{lem:stat-sol-dip}.
\end{proof}

We prove now Theorem \ref{thm:gs-crit}.

\begin{proof}[Proof of Theorem \ref{thm:gs-crit}]
Let $\mu^\star$ be as in the statement of Theorem \ref{thm:gs-crit-dip} and suppose first that $\mu>\mu^\star$. Since $E_\alpha(u)\leq E_0(u)$ for every $u\in\Dtaumu$ and $\E_0(\mu)=-\infty$ if $\mu>\mu^\star$, then $\E_\alpha(\mu)=-\infty$ if $\mu>\mu^\star$. 
Suppose now that $\mu=\mu^\star$. By Theorem \ref{thm:gs-crit-dip}, we know that there exists a Ground State $u$  at mass $\mu^\star$ for $E_0$, satisfying $E_0(u)=0$. By performing mass-preserving transformations $u_\la(x):=\sqrt{\la}u(\la x)$, one gets
\begin{equation*}
    E_\alpha(u_\la)=\la^2 E_0(u)-\la\frac{\alpha}{2}|u(0^-)|^2=-\la\frac{\alpha}{2}|u(0^-)|^2\to -\infty\quad \text{as}\quad\la\to+\infty,
\end{equation*}
entailing that $\E_\alpha(\mu^\star)=-\infty$. We are left to study the case when $0<\mu<\mu^\star$. First of all, observe that, given $v\in \Dtaumu$ with $v(0^-)\neq 0$, by performing mass-preserving transformations one gets
\begin{equation}
    E_\alpha(v_\la)=\la^2 E_\alpha(v)-\la\frac{\alpha}{2} |v(0^-)|^2<0
\end{equation}
for $\la$ sufficiently small, thus $\E_\alpha(\mu)<0$.
Let now $u_n$ be a minimizing sequence for $E_\alpha$ at mass $\mu$. By applying \eqref{eq:gn-crit} and \eqref{gninf-dtau} and using the definition \eqref{eq:mu-star} of $\mu^\star$, there results 
\begin{equation*}
    0>E_\alpha(u_n)\geq \frac{1}{2}\left(1-\frac{\mu}{\mu^\star}\right)\|u_n'\|_2^2-\frac{\alpha}{2}C\sqrt{\mu}\|u_n'\|_2,
\end{equation*}
which entails that $u_n$ is bounded in $H^1(\R^-)\oplus H^1(\R^+)$. Therefore, there exists $u\in H^1(\Rm)\oplus H^1(\Rp)$ such that, up to subsequences, $u_n\deb u$ in $H^1(\Rm)\oplus H^1(\Rp)$. This entails $u_n\deb u$ in $L^\infty_{\text{loc}}(\R)$, in particular $u_n(0^\pm)\to u(0^\pm)$, so that $u\in \Dtau$. By weak lower semicontinuity of the norm, we have $m:=\|u\|_{L^2(\R)}^2\leq \mu$. Suppose first that $m=0$, i.e. $u\equiv 0$. In this case, $u_n(0^-)\to 0$ as $n\to+\infty$, thus by Theorem \ref{thm:gs-crit-dip}
\begin{equation*}
0>\E_\alpha(\mu)=\lim_n E_\alpha(u_n)=\lim_n E_0(u_n)\geq 0,   
\end{equation*}
which is a contradiction, hence $m>0$. Suppose instead that $0<m<\mu$. By applying Brezis-Lieb Lemma \cite{bl}, we obtain that
\begin{equation}
\label{eq:bl-Ealpha}
E_\alpha(u_n)=E_\alpha(u_n-u)+E_\alpha(u)+o(1),\quad n \to+\infty.
\end{equation}
Since $u_n-u\deb 0$ in $H^1(\Rm)\oplus H^1(\Rp)$ and $u_n(0^-)\to u(0^-)$ as $n\to+\infty$, then by Lemma \ref{lem:gn-class-const} and the fact that $\|u_n-u\|_2^4\leq \mu^\star$ there results that
\begin{equation*}
    \liminf_n E_\alpha(u_n-u)\geq 0,
\end{equation*}
that, together with \eqref{eq:bl-Ealpha}, leads to
\begin{equation}
\label{eq:Eal<Epsal}
    E_\alpha(u)\leq \liminf_n E_\alpha(u_n)=\E_\alpha(u).
\end{equation}

On the other hand, there exists $\beta>1$ such that $\beta u\in \Dtaumu$ and, using \eqref{eq:Eal<Epsal}, it holds
\begin{equation*}
   \E_\alpha(\mu)\leq E_\alpha(\beta u)=\beta^2\frac{1}{2}\|u'\|_2^2-\beta^6\frac{1}{6}\|u\|_6^6-\beta^2\frac{\alpha}{2}|u(0^-)|^2<\beta^2 E_\alpha(u)\leq \E_\alpha(u),
\end{equation*}
which is a contradiction, hence $m=\mu$ and $u$ is a Ground State of $E_\alpha$ at mass $\mu$. 

To conclude the proof of Theorem \ref{thm:gs-crit}, it is sufficent to observe that the existence of the further positive stationary solution is guaranteed by the analysis of stationary solutions in Proposition \ref{prop:stat-crit}.
    
\end{proof}

\section{Stationary States}
\label{sec:stat}

In the present Section, we present some results about the stationary states of $E_\alpha$, namely the critical points of \eqref{eq:energy} satisfying the constraint \eqref{eq:mass-const}.

In the first result, we report the equations satisfied by all the stationary states of $E_\alpha$, including Ground States.

\begin{lemma}
\label{lem:stat-el}
Every stationary state $u\in \Dtaumu$ of the energy $E_\alpha$ solves for some $\omega\in\R$ the system
\begin{equation}
\label{eq:bound-state}
\begin{cases}
    u''+|u|^{2\sigma}u=\om u\quad \text{on}\quad \R\setminus\{0\},\quad u\in H^2(\R\setminus\{0\}),\\
    u(0^+)=\tau u(0^-),\\
    u'(0^-)-\tau u'(0^+)=\alpha u(0^-).
\end{cases}
    \end{equation}
\end{lemma}

\begin{proof}
Let $u$ be a stationary state of $E_\alpha$ subject to the constraint $u\in\Dtaumu$. Then by the Lagrange Multiplier Theorem there exists $\omega \in \R$ such that for any $\eta\in \Dtau$ it holds 

\begin{align*}
\intRneg \left(u'\eta'+\omega u \eta-|u|^{2\sigma}u\eta\right)\dx
 +\intRpos \left(u'\eta'+\omega u \eta-|u|^{2\sigma}u\eta\right)\dx  - \alpha u(0^-)\eta(0^-) = 0.
\end{align*}
If we pick $\eta \in C_c ^{\infty}(\R^+)$ or $\eta \in C_c ^{\infty}(\R^-)$, we deduce that $u\in H^2(\R\setminus\{0\})$ and solves the equation $u''+|u|^{2\sigma}u=\omega u$ both on $\Rm$ and $\Rp$. Moreover, by standard elliptic regularity, it is possible to show that $u\in C^2(\R\setminus\{0\})$. For what concerns the conditions at the origin specified in \eqref{eq:bound-state}, the first one holds since $u\in\Dtau$, while for the second we proceed integrating by parts the terms $\int_{\R^\pm}u'\eta' \dx$ and using the equation on $\R^\pm$. Then for any $\eta \in \Dtau$ with $\eta(0^-)\neq 0$ it follows that
\begin{equation*}
u'(0^-)\eta(0^-)-u'(0^+)\eta(0^+) = \alpha u(0^-)\eta(0^-),
\end{equation*}
that, together with the fact that $\eta\in \Dtau$, leads to 
\begin{equation*}
u'(0^-)\eta(0^-)-u'(0^+)\tau\eta(0^-) = \alpha u(0^-)\eta(0^-),    
\end{equation*}
that coincides with the desired condition if $\eta(0^-)\neq 0$.
\end{proof}

\begin{remark}
\label{rem:operator}
Since all stationary states of $E_\alpha$ belong to $H^2(\R\setminus\{0\})$, the system \eqref{eq:bound-state} can be rewritten in a compact form as 
\begin{equation}
\label{eq:stat-Htau}
    H_{\tau,\alpha}u-|u|^{2\sigma}u+\omega u=0,\quad u\in D(H_{\tau,\alpha}),
\end{equation}
where $H_{\tau,\alpha}$ is a self-adjoint operator with domain and action given by 
\begin{equation} \begin{split}
    \label{eq:domain}
D(H_{\tau,\alpha}) & := \left\{ u \in H^2(\R \setminus \{0\}) \, : \, u(0^+)=\tau u(0^-), u'(0^-)-\tau u'(0^+)=\alpha u(0^-)\right\}
\\
H_{\tau,\alpha} u & := -u'',\quad x\neq 0.
\end{split} \end{equation} 

The operator $H_{\tau,\alpha}$ can be obtained as a self-adjoint extension in $L^2(\R)$
of the Laplacian restricted to the set of the functions vanishing in a neighbourhood of the origin (\cite{AN-09,S-86}),
%\begin{equation*}
%-\Delta_{|C^\infty_0(\R\setminus\{0\})},
%\end{equation*}
so by definition (\cite{AGH-KH-88}) it is
 a pointwise perturbation of the Laplacian.
\begin{comment}

 As nowadays widely known, in order to introduce 
       a pointwise 
       perturbation located at the origin of a given coordinate system,  
       one defines an operator that acts as the ordinary Laplacian on functions vanishing at the origin and as a
       perturbed Laplacian on all other functions. Moreover, the 
       perturbation depends on the value of the function at the origin only.
       In this way
       one gets a linear operator to be interpreted as $" - \Delta + \alpha \delta_0"$.

In this context an important distinction arises
between models in one or more dimensions. In fact, to define a linear point perturbation of the Laplacian, the one-dimensional setting  involves a less sophisticated 
mathematical technology and allows for a richer
family of point interactions \cite{S-86,abd}, while in more
dimensions the definition of a point perturbation of 
the Laplacian becomes more technical, involving a
renormalization procedure \cite{MF-1,MF-2} or, equivalently, the search for self-adjoint extensions
of a symmetric operator \cite{AGH-KH-88}. 
\end{comment}
\end{remark}

In the next result we show that there are two important thresholds for the parameter $\omega$:  below the first threshold no solutions exist, between the first and the second threshold only one positive solution exists, while above the second threshold two positive solutions cohexist.

\begin{proposition}
\label{prop:stationary}
Let $\sigma>0$, $\alpha>0$, $\tau>1$ and $\omega\in\R$. Then, denoted by $m$ the multiplicity of the set of positive solutions to \eqref{eq:bound-state} in $L^2(\R)$, there results that
\begin{equation}
    m=\begin{cases}
    0,\quad \omega \leq \frac{\alpha^2}{(\tau^2+1)^2}\\
    1,\quad \frac{\alpha^2}{(\tau^2+1)^2}<\om\leq\frac{\alpha^2}{(\tau^2-1)^2}\\
    2, \quad \omega > \frac{\alpha^2}{(\tau^2-1)^2}
    \end{cases}
\end{equation}

\noindent
Moreover, every solution to \eqref{eq:bound-state} has the form
\begin{equation}
\label{sol}
u(x) = 
\bigg \{ \begin{array}{rl}
\varphi_{\om}(x+x_-), & x \in \R^{-} \\
\varphi_{\om}(x+x_+), & x \in \R^{+} \\
\end{array}
\end{equation}
where $\varphi_{\om}$ is the one dimensional soliton and $x_-, x_+ \in \R$ are the solutions to the system 
\begin{equation}
\label{condition3}
\begin{cases}
\tau^2\sqrt{\om} T_+=\sqrt{\om} T_- + \alpha,\\
\tau^{2\sigma} T_-^2-T_+^2=\tau^{2\sigma}-1,\\
T_\pm=\tanh(\sigma \sqrt{\om} x_\pm).
\end{cases}
\end{equation} 
\end{proposition}

\begin{proof}
It is well known (see e.g. \cite{anv}) that the equation $u''+|u|^{2\sigma}u=\omega u$ has nontrivial solutions in $L^2(\R^\pm)$ if and only if $\omega>0$ and that the only $L^2$ solutions on each half-line are given by $\varphi_\om(x+T)$, where $\varphi_\om$ is given by 
\begin{equation}
\label{soliton}
\varphi_{\om}(x):=\left( \omega (\sigma+1) \right)^{\frac{1}{2\sigma}} \cosh^{-\frac{1}{\sigma}} \left( \sigma \sqrt{\omega} x \right),
\end{equation} 
and $T$ is a suitable real number. Hence, the structure of the solution is given by \eqref{sol}, with $x_\pm$ to be determined, and the problem is reduced to solve the system \eqref{condition3}.

Denote now $T_\pm=\tanh(\sigma \sqrt{\om} x_\pm)$. By the condition $u(0^+)=\tau u(0^-)$, it follows that
\begin{equation*}
\cosh^{-\frac{1}{\sigma}} \left( \sigma \sqrt{\omega} x_+ \right) =\tau \cosh^{-\frac{1}{\sigma}} \left( \sigma \sqrt{\omega} x_- \right),
\end{equation*}
that entails  $(1-T_+^2)^\frac{1}{2\sigma}=\tau(1-T_-^2)^\frac{1}{2\sigma}$, i.e. the second equation in \eqref{condition3}. 
On the other hand, by $u'(0^-)-\tau u'(0^+)=\alpha u(0^-)$ the first equation in \eqref{condition3} follows, thus the proof of \eqref{condition3} is complete.

In order to deduce the multiplicity of the set of positive solutions of \eqref{eq:bound-state}, we are reduced to find the multiplicity of the solutions $(T_-, T_+)\in (-1,1)\times(-1,1)$ to the system 
\begin{equation}
\label{system:t+t-}
\begin{cases}
\tau^2\sqrt{\om} T_+=\sqrt{\om} T_- + \alpha,\\
\tau^{2\sigma} T_-^2-T_+^2=\tau^{2\sigma}-1.
\end{cases}
\end{equation}

By substituting the first equation  of \eqref{system:t+t-} in the second, one obtains that $T_-$ has to solve the equation
\begin{equation*}
    \tau^{2\sigma}T_-^2-\left(\frac{1}{\tau^2}T_-^2+\frac{\alpha}{\tau^2}\sqrt{\om}\right)^2=\tau^{2\sigma}-1,
\end{equation*}
 that can be rewritten as
 \begin{equation}
 \label{eq:T-}
    \left(\tau^{2\sigma+4}-1\right) T_-^2 -2\frac{\alpha}{\sqrt{\omega}}T_- -\left(\frac{\alpha^2}{\omega}+\tau^4(\tau^{2\sigma}-1)\right)=0.
 \end{equation}
 Equation \eqref{eq:T-} has two solutions in $\R$ given by
 \begin{equation*}
 T_-^L =\frac{1}{\tau^{2\sigma+4}-1} \left(\frac{\alpha}{\sqrt{\omega}}-\tau^2 \sqrt{\frac{\alpha^2}{\omega}\tau^{2\sigma}+(\tau^{2\sigma+4}-1)(\tau^{2\sigma}-1)} \right)    
 \end{equation*}
 and
 \begin{equation*}
   T_-^R= \frac{1}{\tau^{2\sigma+4}-1} \left( \frac{\alpha}{\sqrt{\omega}}+\tau^2 \sqrt{\frac{\alpha^2}{\omega}\tau^{2\sigma}+(\tau^{2\sigma+4}-1)(\tau^{2\sigma}-1)} \right). 
 \end{equation*}

In order to check the multiplicity of the set of solutions, it is sufficient to check when $T_-^L$ and $T_-^R$ belong to  $(-1,1)$. In particular, $T_-^L\in(-1,1)$ if and only if $\omega>\frac{\alpha^2}{(\tau^2+1)^2}$, while $T_-^R\in(-1,1)$ if and only if $\omega>\frac{\alpha^2}{(\tau^2-1)^2}$, entailing the thesis.
\end{proof}

\begin{remark}
Let us highlight that system \eqref{system:t+t-} has an easy geometric representation, as shown in Figure \ref{sist}. Indeed, one can observe that the first equation in \eqref{system:t+t-} 
represents a line approaching the origin as $\omega\to+\infty$, but never reaching it since $\alpha \neq 0$. On the other hand, the second equation represents a hyperbola that does not depend on $\om$ and crosses the vertices of the square. The intersections between the line and the hyperbola give us the solutions to \eqref{system:t+t-}.

\end{remark}

\begin{figure}[]
\centering
\includegraphics[width=0.6\columnwidth]{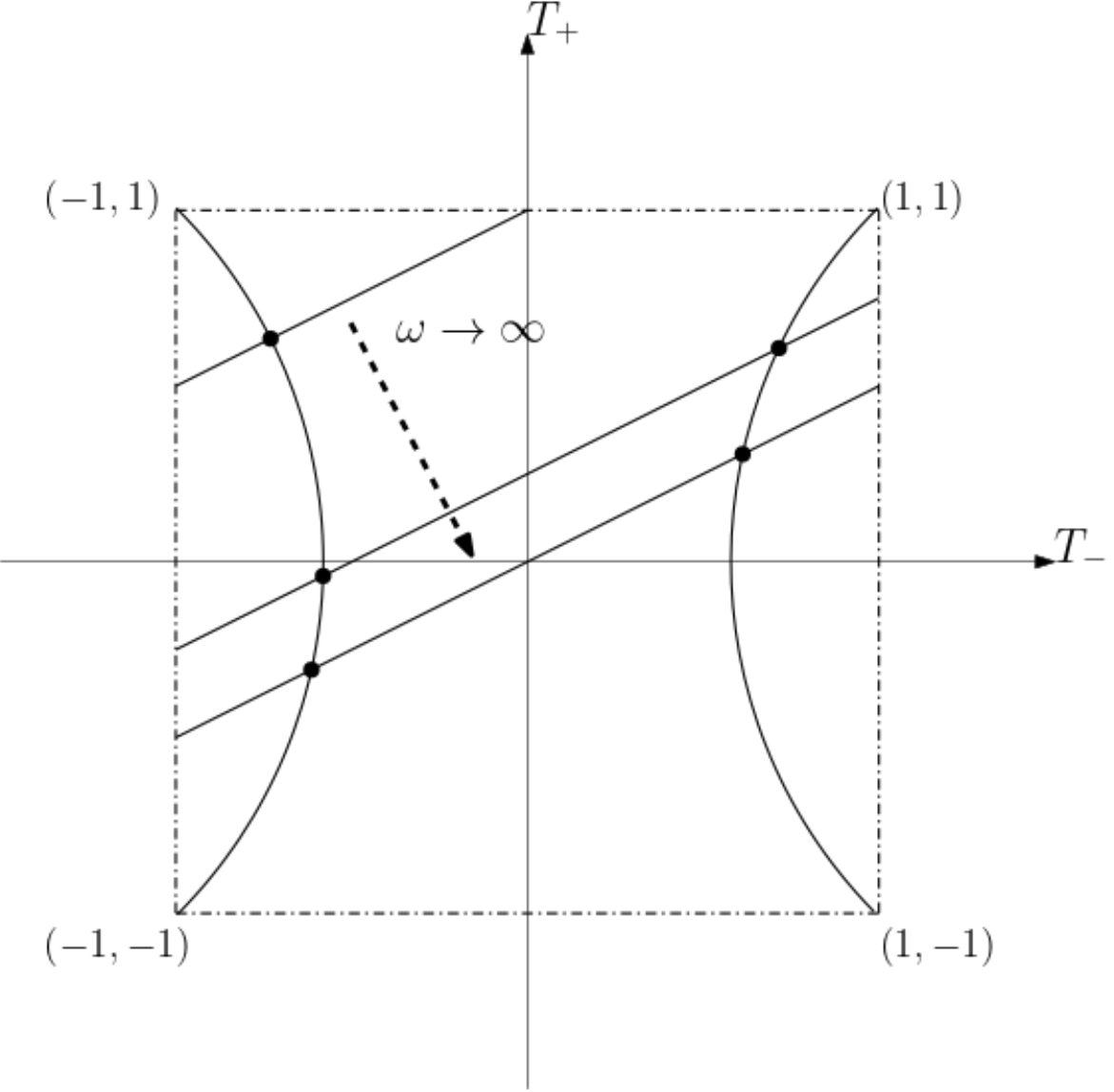}
\caption{Geometric representation of the system \eqref{system:t+t-} for $\tau > 1$, where the dots represent the solutions to the system for $\om \rightarrow \infty$.}
\label{sist}
\end{figure}

%\begin{figure}[H]
%\centering
%\includegraphics[width=0.55\columnwidth]{biforcazione.png}
%\captionof{figure}{Qualitative graph of bifurcation for the stationary states depending on $\om$. Note that the dotted-dashed line refers to the soliton $\phi_\om \notin H_{\tau,\alpha}$.}
%\label{biforcazione}
%\end{figure}

\subsection*{Classification of stationary states} 
 
Let us denote with $u^L_\om$ and $u^R_\om$ the two positive solutions of \eqref{eq:bound-state}. In particular, $u^L_\om$ exists for every $\omega>\frac{\alpha^2}{(\tau^2+1)^2}$, while $u^R_\om$ exists for every $\omega>\frac{\alpha^2}{(\tau^2-1)^2}$. Therefore, for $\frac{\alpha^2}{(\tau^2+1)^2}<\omega\leq\frac{\alpha^2}{(\tau^2-1)^2}$ the only positive solution is given by $u^L$, while for $\omega>\frac{\alpha^2}{(\tau^2-1)^2}$ both $u^L_\om$ and $u^R_\om$ are solutions to \eqref{eq:bound-state}.   

The associated solutions to the system \eqref{system:t+t-}, denoted respectively by $\left(T_-^L,T_+^L, x_-^L, x_+^L\right)$ and $\left(T_-^R,T_+^R, x_-^R,x_+^R\right)$, are given by
\begin{equation*}
\begin{cases}
T_-^L &=\frac{1}{\tau^{2\sigma+4}-1} \left(\frac{\alpha}{\sqrt{\omega}}-\tau^2 \sqrt{\frac{\alpha^2}{\omega}\tau^{2\sigma}+(\tau^{2\sigma+4}-1)(\tau^{2\sigma}-1)} \right)\\
T_+^L &=\frac{1}{\tau^{2\sigma+4}-1} \left(\tau^{2\sigma+2}\frac{\alpha}{\sqrt{\omega}}-\sqrt{\frac{\alpha^2}{\omega}\tau^{2\sigma}+(\tau^{2\sigma+4}-1)(\tau^{2\sigma}-1)} \right)\\
x_-^L &=\frac{1}{\sigma\sqrt{\om}}\arctanh\left(T_-^L\right)\\
x_+^L &=\frac{1}{\sigma\sqrt{\om}}\arctanh\left(T_+^L\right)
\end{cases}
\end{equation*}
and
\begin{equation*}
\begin{cases}
T_-^R &=\frac{1}{\tau^{2\sigma+4}-1} \left(\frac{\alpha}{\sqrt{\omega}}+\tau^2 \sqrt{\frac{\alpha^2}{\omega}\tau^{2\sigma}+(\tau^{2\sigma+4}-1)(\tau^{2\sigma}-1)} \right)\\
T_+^R &=\frac{1}{\tau^{2\sigma+4}-1} \left(\tau^{2\sigma+2}\frac{\alpha}{\sqrt{\omega}}+\sqrt{\frac{\alpha^2}{\omega}\tau^{2\sigma}+(\tau^{2\sigma+4}-1)(\tau^{2\sigma}-1)} \right)\\
x_-^R &=\frac{1}{\sigma\sqrt{\om}}\arctanh\left(T_-^R\right)\\
x_+^R &=\frac{1}{\sigma\sqrt{\om}}\arctanh\left(T_+^R\right)
\end{cases}
\end{equation*}

In particular, let us observe that $x_-^L<0$, $x_-^R>0$ and $x_+^R>0$ for every $\omega$ for which the corresponding solution exists, while $x_+^L<0$ when $\frac{\alpha^2}{(\tau^2+1)^2}<\om<\frac{\tau^{2\sigma}}{\tau^{2\sigma}-1}$ and $x_+^L\geq 0$ when $\om\geq \frac{\tau^{2\sigma}}{\tau^{2\sigma}-1}$. See Figure \ref{st1} and \ref{st2} for a qualitative behaviour of $u^L_\om$ and $u^R_\om$.

\begin{figure}[H]
\centering
\begin{tikzpicture}[xscale= 0.5,yscale=1.5]
    \draw[step=2,thin] (-7,0) -- (7,0);
\node at (0,0) [nodo] {};
\node at (0,-0.3) [] {$0$};
\draw[dashed,thin] (-8.2,0)--(-7.2,0)  (7.2,0)--(8.2,0); 
\draw[color=black] plot[samples=100,domain=-7:0](\x, {2 / cosh(\x - 2)  });
\draw[color=black] plot[samples=100,domain=0:7](\x, {2 / cosh(\x + 1.2) });
\end{tikzpicture}
\end{figure}

\begin{figure}[H]
\begin{tikzpicture}[xscale= 0.5,yscale=1.5]
    \draw[step=2,thin] (-7,0) -- (7,0);
\node at (0,0) [nodo] {};
\node at (0,-0.3) [] {$0$};
\draw[dashed,thin] (-8.2,0)--(-7.2,0)  (7.2,0)--(8.2,0); 
\draw[color=black] plot[samples=100,domain=-7:0](\x, {2 / cosh(\x - 1.4)  });
\draw[color=black] plot[samples=100,domain=0:7](\x, {2 / cosh(\x ) });

\end{tikzpicture}
\end{figure}

\begin{figure}[H]
\centering
\begin{tikzpicture}[xscale= 0.5,yscale=1.5]
    \draw[step=2,thin] (-7,0) -- (7,0);
\node at (0,0) [nodo] {};
\node at (0,-0.3) [] {$0$};
\draw[dashed,thin] (-8.2,0)--(-7.2,0)  (7.2,0)--(8.2,0); 
\draw[color=black] plot[samples=100,domain=-7:0](\x, {2 / cosh(\x - 2.4)  });
\draw[color=black] plot[samples=100,domain=0:7](\x, {2 / cosh(\x - 1.7) });
\end{tikzpicture}
\caption{A sketch of the stationary state $u^L$, in three different qualitative situations. It has always the profile of a tail of a soliton on the negative half-line, whereas on the positive half-line, depending on $\om$, it can be a tail, a half-soliton or presents a bump.}
\label{st1}

\end{figure}
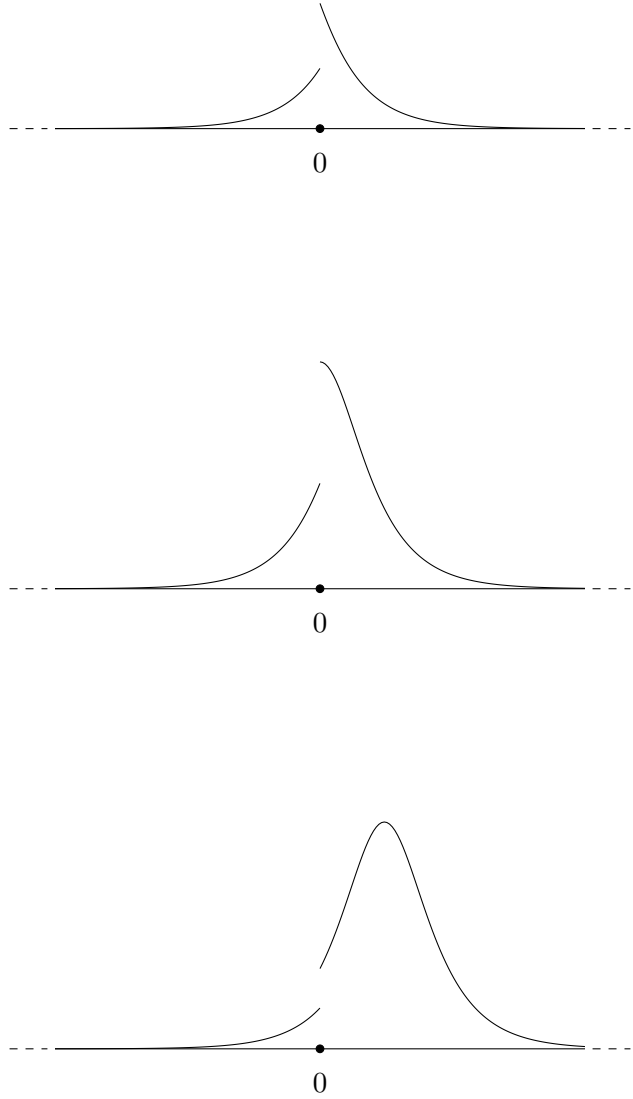

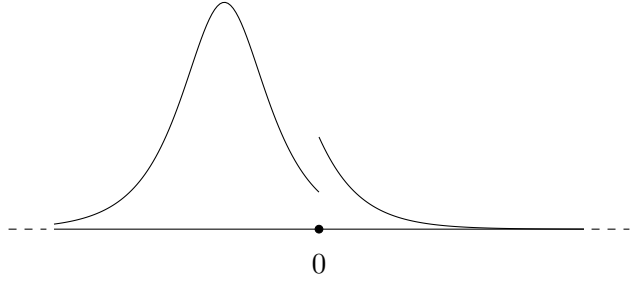
\begin{figure}[H]
\centering
\begin{tikzpicture}[xscale= 0.5,yscale=1.5]
\draw[step=2,thin] (-7,0) -- (7,0);
\node at (0,0) [nodo] {};
\node at (0,-0.3) [] {$0$};
\draw[dashed,thin] (-8.2,0)--(-7.2,0)  (7.2,0)--(8.2,0); 
\draw[color=black] plot[samples=100,domain=-7:0](\x, {2 / cosh(\x + 2.5)  });
\draw[color=black] plot[samples=100,domain=0:7](\x, {2 / cosh(\x + 1.55) });
\end{tikzpicture}

\caption{A sketch of the stationary state $u^R$. This stationary state, regardless of $\om$, has always the profile of a tail of a soliton on the positive half-line and has a bump on the negative one.}
\label{st2}
\end{figure}

\subsection{Identification of the Ground State}

In this subsection, we identify the Ground State of \eqref{eq:energy} at fixed mass $\mu$ among the stationary states $u^L_\om$ and $u^R_\om$.

Let us start with a preliminary lemma. 

\begin{lemma}
Consider $\om>\frac{\alpha^2}{(\tau^2+1)^2}$ and let $u_\om$ be a solution of \eqref{eq:bound-state}. The following identities hold:
\begin{equation}
\label{eq:mass-stat}
    \|u_\omega\|^2_2 = \frac{(\sigma+1)^{\frac{1}{\sigma}}}{\sigma} \om^{\frac{1}{\sigma}-\frac{1}{2}} \left( \int_{-1}^{1} (1-t^2)^{\frac{1}{\sigma}-1}\dt - \int_{T_-(\omega)}^{T_+(\omega)}(1-t^2)^{\frac{1}{\sigma}-1}\,dt \right)
    \end{equation}
    and
\begin{equation}
\label{eq:mass-der}
\begin{split}
    \frac{d}{d\omega}||u_\omega||^2_2=&\frac{(\sigma+1)^{\frac{1}{\sigma}}(2-\sigma)}{2\sigma^2}\omega^{\frac{1}{\sigma}-\frac{3}{2}}\left( \int_{-1}^{1} (1-t^2)^{\frac{1}{\sigma}-1}\dt - \int_{T_-(\omega)}^{T_+(\omega)}(1-t^2)^{\frac{1}{\sigma}-1}\,dt \right)\\
    &+\frac{(\sigma+1)^{\frac{1}{\sigma}}}{\sigma} \om^{\frac{1}{\sigma}-2}\frac{(\tau^{2\sigma}-1)\left(1-T_-(\omega)^2\right)^{\frac{1}{\sigma}-1}}{2\tau^{2\sigma}}\left(2\omega^{\frac{3}{2}}T_-'(\omega)+\frac{\alpha}{\tau^{2\sigma}-1}\right).
    \end{split}
\end{equation}
\end{lemma}
\begin{proof}
By using \eqref{sol} and \eqref{condition3}, it follows that
\begin{equation*}
\begin{split}
  \|u_\omega\|^2_2 &=\int_{-\infty}^{x_-}|\varphi_\omega(x)|^2\,dx+\int_{x_+}^{+\infty}|\varphi_\omega(x)|^2\,dx\\
  &=(\sigma+1)^{\frac{1}{\sigma}}\omega^{\frac{1}{\sigma}}\left(\int_{-\infty}^{x_-}\left(1-\tanh^2(\sigma\sqrt{\om}x)\right)^{\frac{1}{\sigma}}\,dx+\int_{x_+}^{+\infty}\left(1-\tanh^2(\sigma\sqrt{\om}x)\right)^{\frac{1}{\sigma}}\,dx\right)\\
  &=\frac{(\sigma+1)^{\frac{1}{\sigma}}}{\sigma}\omega^{\frac{1}{\sigma}-\frac{1}{2}}\left(\int_{-1}^1\left(1-t^2\right)^{\frac{1}{\sigma}-1}\,dt-\int_{T_-}^{T_+}\left(1-t^2\right)^{\frac{1}{\sigma}-1}\,dt\right).
  \end{split}
\end{equation*}
Computing the derivative of \eqref{eq:mass-stat}, we get
\begin{equation*}
\begin{split}
    \frac{d}{d\omega}\|u_\omega\|^2_2=&\frac{(\sigma+1)^{\frac{1}{\sigma}}(2-\sigma)}{2\sigma^2}\left( \int_{-1}^{1} (1-t^2)^{\frac{1}{\sigma}-1}\dt -\int_{T_-}^{T_+}\left(1-t^2\right)^{\frac{1}{\sigma}-1}\,dt\right)\\
    &-\frac{(\sigma+1)^{\frac{1}{\sigma}}}{\sigma} \om^{\frac{1}{\sigma}-\frac{1}{2}}\left(\left(1-T_+^2\right)^{\frac{1}{\sigma}-1}T_+'-\left(1-T_-^2\right)^{\frac{1}{\sigma}-1}T_-'\right).
    \end{split}
\end{equation*}
Using the second equation in \eqref{condition3} and observing that $T_+'=\frac{1}{\tau^2}\left(T_-'-\frac{\alpha}{2}\omega^{-\frac{3}{2}}\right)$, it holds that
\begin{equation*}
 \left(1-T_+^2\right)^{\frac{1}{\sigma}-1}T_+'-\left(1-T_-^2\right)^{\frac{1}{\sigma}-1}T_-'=-\frac{(\tau^{2\sigma}-1)\left(1-T_-^2\right)^{\frac{1}{\sigma}-1}}{2\tau^{2\sigma}\omega^{\frac{3}{2}}}\left(2\omega^{\frac{3}{2}}T_-'+\frac{\alpha}{\tau^{2\sigma}-1}\right),   \end{equation*}
 which entails \eqref{eq:mass-der}.
\end{proof}
%\begin{equation*}
%\begin{cases}
%(T_-^L)' &=-\frac{\alpha }{2\omega^\frac{3}{2}\left(\tau^{2\sigma+4}-1\right)} \left(1-\frac{\alpha\tau^2}{\sqrt{\alpha^2\tau^{2\sigma}+(\tau^{2\sigma+4}-1)(\tau^{2\sigma}-1)\omega}} \right)\\
%(T_-^R)' &=-\frac{\alpha }{2\omega^\frac{3}{2}\left(\tau^{2\sigma+4}-1\right)} \left(1+\frac{\alpha\tau^2}{\sqrt{\alpha^2\tau^{2\sigma}+(\tau^{2\sigma+4}-1)(\tau^{2\sigma}-1)\omega}} \right)
%\end{cases}
%\end{equation*}

 The next two results, about respectively the subcritical and the critical case, give the multiplicity of the stationary states.

\begin{proposition}
\label{prop:stat-sub}
Defined $\omega_\alpha:=\frac{\alpha^2}{(\tau^2-1)^2}$ and  denoted by $\mu_\alpha$ 
the mass of the soliton of parameter $\omega_\alpha$,
    fix $0<\sigma<2$, $\tau>1$ and $\mu>0$.
    
    If $\mu\leq \mu_\alpha$, then the only positive stationary state at mass $\mu$ is given by $u^L_\om$ and is the Ground State. If  $\mu>\mu_\alpha$, then there are two positive stationary states given by $u^L_\om$ and $u^R_\om$. Among them, the Ground State is given by $u^L_\om$.
\end{proposition}

\begin{proof}
We split the proof in
three steps.

\emph{Step 1. The function $\omega\mapsto \|u_\omega^L\|_2^2$ is bijective from $\left(\frac{\alpha^2}{(\tau^2+1)^2},+\infty\right)$ to $(0,+\infty)$.}

    Let us consider the stationary state $u^L_\om$, that exists for every $\om>\frac{\alpha^2}{(\tau^2+1)^2}$, and observe first that $T_-^L(\om)\to -1$ and $T_+^L(\om) \to 1$ as $\om\to \left(\frac{\alpha^2}{(\tau^2+1)^2}\right)^+$, thus by \eqref{eq:mass-stat} 
    \begin{equation*}
    \left\|u^L_\om\right\|_{2}^2\to 0\quad \text{as} \quad\omega\to \left(\frac{\alpha^2}{(\tau^2+1)^2}\right)^+.
    \end{equation*}
    Moreover
    \begin{equation*}
        \lim_{\omega\to+\infty}T_-^L(\omega)\in(-1,0)\quad \text{and} \quad \lim_{\omega\to+\infty}T_+^L(\omega)\in(-1,0),
    \end{equation*}
    hence by \eqref{eq:mass-stat}
    \begin{equation*}
       \left\|u^L_\om\right\|_{2}^2\to +\infty,\quad \text{as}\quad \om\to+\infty.
    \end{equation*}
    If we prove that $\frac{d}{d\omega}\left\|u^L_\omega\right\|^2_2>0$, then for every $\mu>0$ there exists an only $\omega>\frac{\alpha^2}{(\tau^2+1)^2}$ such that $u^L_\om\in \Dtaumu$. To such aim, let us observe that by \eqref{eq:mass-der} and recalling that
    \begin{equation*}
        \int_{-1}^{1} (1-t^2)^{\frac{1}{\sigma}-1}\dt - \int_{T_-^L(\omega)}^{T_+^L(\omega)}(1-t^2)^{\frac{1}{\sigma}-1}\,dt>0\quad \forall\, \om>\frac{\alpha^2}{(\tau^2+1)^2},
    \end{equation*}
    one has $\frac{d}{d\omega}\|u^L_\omega\|^2_2>0$ if $2\omega^{\frac{3}{2}}\left(T_-^L\right)'(\omega)+\frac{\alpha}{\tau^{2\sigma}-1}>0$. By computing the derivative of $T_-^L$, one gets
    \begin{equation*}
       \left(T_-^L\right)'(\om) =-\frac{\alpha }{2\omega^\frac{3}{2}\left(\tau^{2\sigma+4}-1\right)} \left(1-\frac{\alpha\tau^{2\sigma+2}}{\sqrt{\alpha^2\tau^{2\sigma}+(\tau^{2\sigma+4}-1)(\tau^{2\sigma}-1)\omega}} \right),
    \end{equation*}
    thus $2\omega^{\frac{3}{2}}\left(T_-^L\right)'(\omega)+\frac{\alpha}{\tau^{2\sigma}-1}>0$ if and only if
    \begin{equation*}
\left(\frac{\tau^{2\sigma+4}-1}{\tau^{2\sigma}-1}-1\right)+\frac{\alpha\tau^{2\sigma+2}}{\sqrt{\alpha^2\tau^{2\sigma}+(\tau^{2\sigma+4}-1)(\tau^{2\sigma}-1)\omega}}>0,
    \end{equation*}
    which is satisfied since both the terms of the left-hand side are positive.
    
    \emph{Step 2. The function $\omega\mapsto \|u_\omega^R\|_2^2$ is bijective from $\left(\frac{\alpha^2}{(\tau^2-1)^2},+\infty\right)$ to $(\mu_\alpha,+\infty)$.}

    Let us consider now the stationary state $u^R_\om$, that exists for every $\om>\om_\alpha$, and observe that $T_-^R(\om)\to 1$ and $T_+^R(\om)\to 1$ as $\omega\to (\om_\alpha)^+$, hence by \eqref{eq:mass-stat} $\left\|u^R_\om\right\|_2^2\to \mu_\alpha$ as $\om\to (\om_\alpha)^+$. Moreover, since
    \begin{equation*}
       \lim_{\omega\to+\infty}T_-^R(\omega)\in(0,1)\quad \text{and} \quad \lim_{\omega\to+\infty}T_+^R(\omega)\in(0,1), 
    \end{equation*}
    by \eqref{eq:mass-stat} it follows that $\left\|u^R_\om\right\|_2^2\to +\infty$ as $\om\to+\infty$. Arguing as in Step $1$, in order to conclude it is sufficient to prove that
    \begin{equation}
        \label{eq:2om32}
        2\omega^{\frac{3}{2}}(T_-^R)'(\om)+\frac{\alpha}{\tau^{2\sigma}-1}>0.
    \end{equation}
    By direct computations, it holds
    \begin{equation*}
     (T_-^R)'(\om)=-\frac{\alpha }{2\omega^\frac{3}{2}\left(\tau^{2\sigma+4}-1\right)} \left(1+\frac{\alpha\tau^{2\sigma+2}}{\sqrt{\alpha^2\tau^{2\sigma}+(\tau^{2\sigma+4}-1)(\tau^{2\sigma}-1)\omega}} \right),  
    \end{equation*}
    hence arguing as in Step 1 condition \eqref{eq:2om32} holds if and only if
    \begin{equation}
     \frac{\alpha\tau^2}{\sqrt{\alpha^2\tau^{2\sigma}+(\tau^{2\sigma+4}-1)(\tau^{2\sigma}-1)\omega}}<\frac{\tau^4-1}{\tau^{2\sigma}-1}.  
    \end{equation}
    By taking the square of both sides of the inequality and making some computations, we get
    \begin{equation*}
        (\tau^{4}-1)^2(\tau^{2\sigma+4}-1)(\tau^{2\sigma}-1)\omega>\alpha^2\left[\tau^4(\tau^{2\sigma}-1)^2-\tau^{2\sigma}(\tau^4-1)^2\right].
    \end{equation*}
    Since $\tau > 1$, the l.h.s. is  positive, while the r.h.s. is negative, hence \eqref{eq:2om32} is  satisfied and Step 2 is concluded.
    
    \emph{Step 3. Identification of the Ground State.} By Step 1 and Step 2, if $0<\mu\leq \mu_\alpha$, then the only stationary state at mass $\mu$ is given by the state $u^L_\om$ at some frequency $\omega>\frac{\alpha^2}{(\tau^2+1)^2}$: since a Ground State always exists by Theorem \ref{thm:gs-sub}, $u^L_\om$ is also the Ground State of $E_\alpha$ at mass $\mu$. If instead $\mu>\mu_\alpha$, then there exist $\omega_1$ and $\omega_2$ such that $u^L_{\om_1}\in \Dtaumu$ and $u^R_{\om_2}\in \Dtaumu$. By Lemma \ref{lem:rearr}, since the function $u^R_{\om_2}$ has at least $3$ preimages for almost every $t\in \left(u^R_{\om_2}(0^-),u^R_{\om_2}(0^+)\right)$, there results that there exists $u^\star\in \Dtaumu$ such that $E_\alpha(u^\star)<E_\alpha\left(u^R_{\om_2}\right)$, entailing that $u^R_{\om_2}$ is not a Ground State of $E_\alpha$ at mass $\mu$. Therefore, the Ground State is given by $u^L_{\omega_1}$.

\end{proof}

\begin{proposition}
\label{prop:stat-crit}
Let $\sigma=2$, $\tau>1$ and $\mu^\star$ as in \eqref{eq:mu-star-exact}. Therefore, if $0<\mu<\mu^\star$, then the only positive stationary state at mass $\mu$ is given by
$u^L_\om$. If instead $\frac{\sqrt{3}}{2}\pi<\mu<\widetilde{\mu}$, then the only positive stationary state at mass $\mu$ is given by $u^R_\omega$.  
\end{proposition}

\begin{proof}
We split the proof in two steps.

\emph{Step 1. The function $\omega\mapsto \|u_\omega^L\|_2^2$ is bijective from $\left(\frac{\alpha^2}{(\tau^2+1)^2},+\infty\right)$ to $(0,\mu^\star)$.} As in Proposition \ref{prop:stat-sub}, $T_-^L(\omega)\to -1$ and $T_+^L(\omega)\to 1$ as $\omega\to \left(\frac{\alpha^2}{(\tau^2+1)^2}\right)^+$, thus $\|u_\omega^L\|_2^2\to 0$. Moreover
\begin{equation*}
T_-^L(\omega)\to -\frac{\tau^2}{\sqrt{\tau^4+1}},\quad T_+^L(\omega)\to -\frac{1}{\sqrt{\tau^4+1}},\quad \text{as}\quad \omega\to+\infty.
\end{equation*}
We notice that, in view of Lemma \ref{lem:stat-sol-dip}, we can deduce that $x_-^L(\omega)\to -\xi_-$ and $x_+^L(\omega)\to \xi_+$, with $\xi_\mp$ as in Lemma \ref{lem:stat-sol-dip}. Therefore, since the translations $-\xi_-$ and $-\xi_+$ identify the stationary state $u_1$ of Lemma \ref{lem:stat-sol-dip}, by combining \eqref{eq:mass-u1-u2} and \eqref{eq:mu-star-exact} we deduce that  
\begin{equation*}
\|u_\omega^L\|_2^2\to \mu^\star\quad \text{as}\quad \omega\to+\infty,
\end{equation*}
with $\mu^\star$ as in \eqref{eq:mu-star-exact}. If we prove that $\|u_\omega^L\|_2^2$ is a strictly increasing function of $\omega$, then Step 1 follows. Since $\sigma=2$, by \eqref{eq:mass-der} the function $\omega\mapsto \|u_\omega^L\|_2^2$ is strictly increasing if and only if $2\omega^{\frac{3}{2}}(T_-^R)'(\om)+\frac{\alpha}{\tau^{4}-1}>0$. Repeating the same arguments in Proposition \ref{prop:stat-sub}, we conclude Step 1.

\emph{Step 2. The function $\omega\mapsto \|u_\omega^L\|_2^2$ is bijective from $\left(\frac{\alpha^2}{(\tau^2-1)^2},+\infty\right)$ to $\left(\frac{\sqrt{3}}{2}\pi,\widetilde{\mu}\right)$.} Similarly to Step 1, it is easy to check that $T_-^R(\omega)\to 1$, $T_+^R(\omega)\to 1$ and $\|u_\omega^R\|_2^2\to \frac{\sqrt{3}}{2}\pi$ as $\omega \to \left(\frac{\alpha^2}{(\tau^2-1)^2}\right)^+$. On the other hand, it holds
\begin{equation*}
T_-^R(\omega)\to \frac{\tau^2}{\sqrt{\tau^4+1}}, \quad T_+^R(\omega)\to \frac{1}{\sqrt{\tau^4+1}},\quad \text{as}\quad \mu\to+\infty. 
\end{equation*}
Arguing as in Step 1, one deduces  $x_\mp^R\to \xi_\mp$, where the translations $\xi_\mp$ are as in Lemma \ref{lem:stat-sol-dip} and identify the stationary state $u_2$ (see again Lemma \ref{lem:stat-sol-dip}). By \eqref{eq:mass-u1-u2} and \eqref{eq:mu-tilde}. Then it follows  $\|u_\omega^R\|_2^2\to \widetilde{\mu}$ as $\omega \to+\infty$. Since one can prove that $\|u_\omega^R\|_2^2$ is strictly increasing as in Proposition \ref{prop:stat-sub}, we conclude the proof of Step 2 and, together with Step 1, the thesis follows.
\end{proof}

\appendix

\section{A modified Gagliardo-Nirenberg inequality}

We report here a modified Gagliardo-Nirenberg inequality that has not been used along the paper, but it could be interesting by itself.

\begin{lemma}
\label{lem:mod-gn-ineq}
Let $\tau>1$. Then for every $u\in \Dtau$ there results
\begin{equation}
\label{eq:mod-gn-tau}
    \|u\|_{L^6(\R)}^6+\left(\tau^8-1\right)\|u\|_{L^6(\Rm)}^6\leq \frac{4}{\pi^2}\left(\|u\|_{L^2(\R)}^2+\left(\tau^4-1\right)\|u\|_{L^2(\Rm)}^2\right)^2\|u'\|_2^2\quad \forall\,u\in \Dtau.
\end{equation}
\end{lemma}
\begin{proof}
Let $u\in \Dtau$ and define the function $u_\tau$ as
\begin{equation*}
    u_\tau(x):=
    \begin{cases}
      u(x),\quad &x\in \Rp,\\
      \tau u\left(\tau^{-2}x\right),\quad &x\in \Rm.
    \end{cases}
\end{equation*}
Since $u_\tau(0^-)=\tau u(0^-)=u(0^+)=u_\tau(0^+)$, then $u_\tau\in H^1(\R)$. Therefore, by applying \eqref{gnp} for $\sigma=2$ and recalling that the optimal constant in this case is given by $\frac{4}{\pi^2}$ (see for example \cite{C-03}), it follows that
 \begin{equation}
 \label{eq:gn-utau}
 \|u_\tau\|_{L^6(\R)}^6\leq \frac{4}{\pi^2}\|u_\tau'\|_{L^2(\R)}^2\|u_\tau\|_{L^2(\R)}^4.
 \end{equation}
 Inequality \eqref{eq:mod-gn-tau} follows by \eqref{eq:gn-utau} observing that
 \begin{equation*}
\|u_\tau'\|_{L^2(\R)}^2=\|u'\|_{L^2(\R)}^2\quad\text{and}\quad\|u_\tau\|_{L^p(\R)}^p=\|u\|_{L^p(\Rp)}^p+\tau^{p+2}\|u\|_{L^p(\Rm)}^p\quad\forall\, p\geq 1.
\end{equation*}
\end{proof}

\subsection*{Acknowledgments}

R.A. acknowledges that this study was carried out within the project E53D23005450006 "Nonlinear dispersive equations in presence of singularities" - funded by European Union - Next Generation EU within the PRIN 2022 program (D.D. 104 - 02/02/2022 Ministero dell’Università e della Ricerca). This manuscript reflects only the authors' views and opinions and the Ministry cannot be considered responsible for them.

R.A. was supported by the INdAM programme 
"Analisi spettrale, armonica e stocastica in presenza di  potenziale magnetici", funded in
the framework
"Progetti di Ricerca 2025".

F.B. has been partially supported by  the INdAM Gnampa 2023 project "Modelli nonlineari in presenza di interazioni puntuali".

R.A. and T.N. are grateful to Prof. Taksu Cheon for first suggesting that the interplay of
Fülöp-Tsutsui point interaction and nonlinearity could lead to an interesting phenomenology.

\end{document}